\documentclass{cmslatex}
\usepackage[paperwidth=7in, paperheight=10in, margin=.875in]{geometry}
\usepackage[backref,colorlinks,linkcolor=red,anchorcolor=green,citecolor=blue]{hyperref}

\usepackage{paper}
\usepackage{todonotes} 
\usepackage{comment}
\usepackage{bbm}

\newcommand{\triu}[1]{\mathrm{triu}\left(#1\right)}

\newcommand{\ie}{\latin{i.e.}}
\newcommand{\eg}{\latin{e.g.}}

\newcommand{\tx}{\widetilde{x}}
\newcommand{\ty}{\widetilde{y}}

\newcommand{\ttheta}{\widetilde{\theta}}
\newcommand{\veps}{\varepsilon}
\renewcommand{\tr}{\widetilde{r}}
\newcommand{\tA}{\widetilde{A}}

\newcommand{\TOM}{TriOFM}
\newcommand{\OFM}{OFM}
\newcommand{\TOMs}{TriOFMs}
\newcommand{\OFMs}{OFMs}

\begin{document}

\title{Global Convergence of Triangularized Orthogonalization-free Method
\thanks{Authors are listed in alphabetical order.}
}

\author{
    Weiguo Gao\thanks{School of Mathematical Sciences, Fudan University,
    Shanghai, China 200433; School of Data Science, Fudan University,
    Shanghai, China 200433 (wggao@fudan.edu.cn)}
    \and Yingzhou Li\thanks{School of Mathematical Sciences, Fudan University,
    Shanghai, China 200433 (yingzhouli@fudan.edu.cn)}
    \and Bichen Lu\thanks{Shanghai Center for Mathematical Sciences, Shanghai,
    China 200438 (bclu18@fudan.edu.cn)}
}
\pagestyle{myheadings} \markboth{Global Convergence of TriOFM}{W. Gao, Y.
Li, B. Lu}
\maketitle

\begin{abstract} 

This paper proves the global convergence of a triangularized
orthogonalization-free method (\TOM{}). \TOM{}, in general, applies
a triangularization idea to the gradient of an objective function and
removes the rotation invariance in minimizers. More precisely, in this
paper, the \TOM{} works as an eigensolver for sizeable sparse matrices
and obtains eigenvectors without any orthogonalization step. Due to the
triangularization, the iteration is a discrete-time flow in a non-conservative
vector field. The global convergence relies on the stable manifold theorem,
whereas the convergence to stationary points is proved in detail in this
paper. We provide two proofs inspired by the noisy power method and the
noisy optimization method, respectively.

\end{abstract}



\reversemarginpar

\section{Introduction}
\label{sec:introduction}

Solving low-lying eigenpairs of a symmetric matrix is the key task in
computational chemistry. However, the orthogonalization step in most
classical iterative eigensolvers, \eg, Lanczos method, Davidson method,
LOBPCG, etc, is too expensive to be carried out explicitly every
iteration. For example, in a chemistry application named full
configuration interaction (FCI), the matrix size grows exponentially as
the system size increases. The matrix size normally ranges from $10^8$ to
$10^{40}$. Explicitly conducting the orthogonalization step is not
affordable. In other chemistry applications like linear-scaling density
functional theory (linear-scaling DFT), the explicit orthogonalization
step would scale cubically and should be avoided. Therefore,
orthogonalization free methods (\OFMs{})~\cite{Corsetti2014, Lu2017a,
Lu2017} play an important role in these applications. While one feature
for \OFMs{} is the rotation-invariant property, \ie, multiplying an
orthogonal matrix on columns of the iteration variable preserves the
underlying objective value. Such a feature often leads to less sparse
minima comparing to the sparsity in the original eigenvectors. Therefore,
the triangularized orthogonalization-free method (\TOM{}) is proposed to
avoid orthogonalization as well as preserving the sparsity in the original
eigenvectors. In this paper, we aim to prove the global convergence of the
\TOM{} for solving the extreme eigenvalue problem.

\TOM{} is a method proposed to modify the gradient to decouple former
columns from later columns. Without using the triangularization technique,
the \OFM{} objective function adopted in this paper is,
\begin{equation}\label{eq:obj}
    f(X) = \fnorm{A+XX^\top}^2,
\end{equation}
where $A \in \bbR^{n \times n}$ is a symmetric matrix of size $n$ by $n$
and $X \in \bbR^{n \times p}$ for $p$ being the number of desired
eigenpairs. Several existing work~\cite{Gao2020, Li2019c, Wen2016} shows
that \eqref{eq:obj} has no spurious local minimizers, and
all local minima are global minima. Hence, gradient-based methods applied
to \eqref{eq:obj} converge to global minima almost surely. The global
convergence proof is a direct combination of the first-order convergence
to stationary points and the stable manifold theorem~\cite{Li2019c}.

When the triangularization technique is enabled, the iterative scheme is
no longer a discrete gradient flow. Instead, it still admits the form of a
discrete dynamical system. As been analyzed in the original \TOM{}
paper~\cite{Gao2020}, the stable fixed points are scaled eigenvectors
corresponding to the smallest $p$ eigenvalues, and all other fixed points
are unstable. Stable manifold theorem, therefore, can still be adopted to
show the global convergence of the method as long as the convergence to
fixed points is proved for \TOM{}. This paper mainly focuses on the
convergence proof to fixed points and then proves the global convergence
of \TOM{} via a recursive application of the stable manifold theorem.

\subsection{Related Work}

\TOM{} as an iterative eigensolver has plenty of related work from
numerical linear algebra. We refer readers to our original \TOM{}
paper~\cite{Gao2020} and reference therein. In this section, we focus on
reviewing work related to the global convergence of non-convex functions.

For general non-convex functions, converging to global minima is a tough
problem. Some algorithms~\cite{Laio2002} could be shown of global
convergence property. While the performance is not ideal for high
dimensional problems. Viewing the eigenvalue problem as an unconstrained
optimization problem, the objective function, as in \eqref{eq:obj}, is
non-convex but has no spurious local minima. That means all local minima
of \eqref{eq:obj} are global minima. Besides the eigenvalue problem, other
matrix factorization or matrix completion problems have been shown to
admit the same property. In~\cite{Ge2016}, authors show that the matrix
completion problem for symmetric matrices has no spurious local minimum.
Later, the result is extended to asymmetric low-rank
problems~\cite{Ge2017}, where a unified geometric analysis is proposed and
has been applied to matrix sensing, matrix completion, and robust
principal component analysis (PCA). Later, \citet{Fattahi2020} proved that
non-negative rank-1 robust PCA has no spurious local minimum.
Recently, \citet{gang2019} applied the triangularized
method to a similar but not identical objective function as ours in the
area of distributed PCA. A linear convergence in the sense of power method
was then proved in \cite{gang2022}.

Given the absence of spurious local minimum, the global convergence is
guaranteed if the algorithm can converge to first-order points and escape
from strict saddle points. Second-order methods naturally avoid strict
saddle points and converge to global minima. First-order methods, however,
avoid strict saddle points almost surely. If the first-order methods
contain additive random noise, \eg, stochastic gradient descent, the noise
could perturb the iteration variable into the decreasing sector around
strict saddle points. When the iteration is stuck around strict saddle
points for a long time, the method is guaranteed to escape from the saddle
points with high probability. Hence global convergence is achievable with
high probability. When the noise is not additive, for example, in the
randomized coordinate descent method, escaping from saddle points is
related to the stability analysis of Lyapunov exponent of random dynamic
system~\cite{Froyland2015, Ledrappier1991} and the proof could be carried
out with conditions on the stepsize~\cite{Chen2021b}. If the first-order
method is noise-free, the global convergence proof is directly related to
the classical stable manifold theorem in the dynamic system. The idea is
recently carried over to the optimization community in machine
learning~\cite{Lee2019}.

\subsection{Our Contributions}

In this paper, we prove the global convergence property of \TOM{}. Since
\TOM{} is not a gradient-based optimization method, the global convergence
is not in the sense of optimization. We show that \TOM{} can escape from
unstable fixed points and converge to stable fixed points, which are the
desired low-lying eigenvectors directly.

The overall strategy in proving the global convergence is to iteratively
apply the stable manifold theorem to each column of the iteration variable
$X$. The major difficulty lies in proving the convergence to a
(stable/unstable) fixed point. The reason behind the difficulty is as
follows. Given a column in $X$, the update across iterations is composed
of two parts: a force term driving towards the fixed point and a
perturbation due to the updates from related other columns. Hence, for the
method to be convergent, the first part must be dominant, and the second
part is well-controlled. In this paper, we propose two strategies to show
this. The first strategy is inspired by the convergence proof of the noisy
power method~\cite{Hardt2014}. The convergence is then decomposed into the
convergence of the angle between $X$ and fixed points and the convergence
of the column lengths of $X$. The second strategy introduces different
energy functions for different columns of $X$. The convergence proof
then is similar to that in optimization with a careful estimation of the
magnitude of the perturbation part. Two strategies have their own unique
features to be generalized for other \TOMs{}. The first strategy is
well-suited for methods whose iterative schemes are analogs to a shifted
and scaled power method. The second strategy can be applied to methods
whose energy functions for different columns are known explicitly.

\subsection{Organization}

In the rest paper, Section~\ref{sec:preliminary} states the \TOM{}
algorithm as well as properties of the stable and unstable fixed points.
Section~\ref{sec:globalconv} gives the global convergence for \TOM{}. Two
proofs for the key lemmas are provided in Section~\ref{sec:proof_details}.
A numerical experiment is given in
Section~\ref{sec:num_res} in facilitating understanding the theoretical
converging behavior.

\section{\TOM{} Revisit}
\label{sec:preliminary}

We revisit the existing results of \TOM{} in this section. Notations that
will be used throughout this paper is shown in Table~\ref{tab:notations}.

\begin{table}[ht]  
    \centering
    \begin{tabular}{lp{0.7\textwidth}}
        \toprule  
        Notation & Explanation \\ 
        \toprule
        $n$ & The size of the matrix.\\
        $q$ & The number of negative eigenvalues of the matrix.\\
        $p$ & The number of desired eigenpairs and $p \leq q$.\\
        \midrule
        $A$ & The $n$-by-$n$ symmetric matrix.\\
        $\lambda_i$ & The $i$-th smallest eigenvalue of $A$.\\
        $\Lambda$ & A diagonal matrix with diagonal entries being
              $\lambda_1, \dots, \lambda_n$.\\
        $\Lambda_i$ & The first $i$-by-$i$ principal submatrix of
              $\Lambda$.\\
        $U$ & An orthogonal matrix satisfying $U^\top A U=\Lambda$.\\
        $u_i$ & The eigenvector of $A$ corresponding to $\lambda_i$.\\
        $U_i$ & The first $i$ columns of $U$.\\
        $\rho$ & The 2-norm of $A$, \ie, $\rho = \norm{A}_2$.\\
        \midrule
        $X^{(t)}$ & An $n$-by-$p$ matrix denoting the vectors at $t$-th iteration.\\
        $x^{(t)}_i$ & The $i$-th column of $X^{(t)}$.\\
        $X^{(t)}_i$ & The first $i$ columns of $X^{(t)}$.\\
        $f(X)$ & The objective function of $X$, \eqref{eq:obj}.\\
              $\grad{f}(X)$ & The gradient of $f(X)$.\\
        \midrule
        $\alpha$ & The stepsize used in optimization algorithms.\\
        $e_i$ & The $i$-th standard basis vector, \ie, a vector of length
              $n$ with one on the $i$-th entry and zero elsewhere. \\
        \bottomrule  
    \end{tabular}
    \caption{Notations.}
    \label{tab:notations}
\end{table}

Recall the gradient of \eqref{eq:obj},
\begin{equation} \label{eq:grad}
    \grad{f}(X) = 4AX + 4XX^\top X.
\end{equation}
The $i$-th column in the first term above is $Ax_i$ which only involves
$x_i$ itself and is independent of other columns of $X$. However, the second
term above mixes all columns of $X$. According to~\cite{Li2019c}, we know
that if $X$ is a single column vector then the gradient descent method
applied to \eqref{eq:obj} converges to the eigenvector corresponding to the
smallest eigenvalue. If $X$ has two columns, then the iteration converges to
the eigenspace spanned by the two eigenvectors corresponding to the smallest
two eigenvalues. These analytical results inspire the design of the \TOM{}
scheme, where the iteration obeys,
\begin{eqnarray}
    x_1^{(t+1)} & = & x_1^{(t)} - \alpha \left( A x_1^{(t)} + x_1^{(t)}
    \left(x_1^{(t)}\right)^\top x_1^{(t)} \right), \label{eq:iter1}\\
    x_2^{(t+1)} & = & x_2^{(t)} - \alpha \left( A x_2^{(t)}
    + x_1^{(t)}
    \left(x_1^{(t)}\right)^\top x_2^{(t)}
    + x_2^{(t)}
    \left(x_2^{(t)}\right)^\top x_2^{(t)} \right), \label{eq:iter2}\\
    & \cdots & \nonumber\\
    x_i^{(t+1)} & = & x_i^{(t)} - \alpha \left( A x_i^{(t)}
    + \sum_{j=1}^{i}x_j^{(t)}
    \left(x_j^{(t)}\right)^\top x_i^{(t)} \right), \label{eq:iteri}
\end{eqnarray}
for $i = 1, 2, \dots, p$. Throughout this paper, the constant 4 in
\eqref{eq:grad} is absorbed into stepsize $\alpha$.

Iterative scheme \eqref{eq:iter1} is a gradient descent method applied to
\eqref{eq:obj} with a single vector. Iterative scheme \eqref{eq:iter2}
admits the updates on the second column of a gradient descent method
applied to \eqref{eq:obj} with two vectors. Similar idea is applied to the
rest vectors. Using the iterative scheme above, we expect that: i) the first
column converges to the eigenvector corresponding to the smallest
eigenvalue; ii) the convergence of other columns can be done inductively.
To simplify the notation, we define a modified updating direction as,
\begin{equation} \label{eq:g}
    g(X) = AX + X \triu{X^\top X},
\end{equation}
where $\triu{\cdot}$ denotes the upper triangular part of a given matrix.
Then the iterative scheme in \eqref{eq:iteri} can be rewritten as,
\begin{equation} \label{eq:triofm}
    X^{(t+1)} = X^{(t)} - \alpha g\left( X^{(t)} \right).
\end{equation}
Here the stepsize $\alpha$ is assumed to be a constant in this paper. Other
strategies for $\alpha$ refer to~\cite{Gao2020}.

The modified updating direction~\eqref{eq:g} is not a gradient of any
energy function, and hence the vector field formed by~\eqref{eq:g} is not
conservative. The stable and unstable stationary points of~\eqref{eq:g}
have been analyzed in~\cite{Gao2020}. For completeness, we include the
conclusion as in Theorem~\ref{thm:fixedpt}.

\begin{theorem}[Theorem 3.1 from \cite{Gao2020}] \label{thm:fixedpt}
    All {\bfseries fixed points} of \eqref{eq:triofm} are of form $X =
    U_q \sqrt{-\Lambda_q} P S$, where $\sqrt{\cdot}$ is applied entry-wise,
    $P \in \bbR^{q \times p}$ is the first $p$ columns of an arbitrary
    $q$-by-$q$ permutation matrix, and $S \in \bbR^{p \times p}$ is a
    diagonal matrix with diagonal entries being $0$ or $\pm 1$. Within these
    points all {\bfseries stable fixed points} are of form $X = U_p
    \sqrt{-\Lambda_p} D$, where $D \in \bbR^{p \times p}$ is a diagonal
    matrix with diagonal entries being $\pm 1$. Others are {\bfseries
    unstable fixed points}.
\end{theorem}

All the stable fixed points are composed of desired eigenvectors, with its
length being the square root of the corresponding eigenvalue. All other
unstable fixed points are also composed of eigenvectors up to a scaler.
Next, we will prove that under mild conditions, the iterative
scheme~\eqref{eq:triofm} converges to the stable fixed points for almost
all initializations.

\section{Global Convergence Analysis of \TOM{}}
\label{sec:globalconv}

In this section, we will prove the global convergence for the fixed
stepsize \TOM{}. The analysis of \TOM{} is intuitively straightforward but
technically challenging.

First, we will explain the intuition behind the convergence of \TOM{}.
Each column of the variable in \TOM{} addresses a different optimization
problem. Let us consider the convergence of the third column. Assume the
first two columns have converged. Since the third column is decoupled from
later columns, the convergence of the third column in \TOM{} is fully
determined by the first three columns. Intuitively, if we assume the first
two columns $x_1, x_2$ are frozen to be a stable fixed point, then the
third column in \eqref{eq:triofm} is associated with an optimization
problem,
\begin{equation} \label{eq:x3}
    \fnorm{A + x_3 x_3^\top}^2 + 2\fnorm{x_1^\top x_3}^2
    + 2\fnorm{x_2^\top x_3}^2 = \fnorm{\tA +
    x_3 x_3^\top}^2,
\end{equation}
where $\tA = A + x_1 x_1^\top + x_2 x_2^\top = A - \lambda_1 u_1 u_1^\top
- \lambda_2 u_2 u_2^\top$. Hence the optimization problem for $x_3$ is of
the same form as the single-column version of \eqref{eq:triofm}, whose
global convergence is guaranteed for almost all
initializations~\cite{Li2019c}. Then, by an induction argument, the
convergence analysis can be applied to later columns one-by-one. However,
there are two niches. First, the numerical error for the converged $x_1$
and $x_2$ should be considered in the convergence analysis of $x_3$, \ie,
the objective function in \eqref{eq:x3} has an extra error term. The
second niche is problematic. For random initialized $x_1$, $x_2$, and
$x_3$, the convergence of $x_1$ and $x_2$ should not make $x_3$ fall into
the problematic zero-measure initial set where $x_3$ will converge to an
unstable stationary point. A complete proof must fulfill these two
niches.

In the following, we present the rigorous global convergence analysis for
\TOM{}. All lemmas and theorems for the convergence are stated under
Assumption~\ref{assump:init}, where we assume the iteration starts in a
big domain. Within the domain, the Hessian matrices are bounded from
above.

\begin{assumption} \label{assump:init}
    Let $R_i=2^{i-1}\sqrt{3\rho}$ for all $1 \leq i\leq p$. The initial
    point $X^{(0)} = \begin{pmatrix}x_1^{(0)} & x_2^{(0)} & \cdots &
    x_p^{(0)}\end{pmatrix}$ satisfies $\norm{x_i^{(0)}} \leq R_i$ for all
    $1\leq i\leq p$. The stepsize in~\eqref{eq:triofm} satisfies $\alpha
    \leq \frac{1}{10R_p^2}$.
\end{assumption}

According to Theorem~\ref{thm:fixedpt}, stable fixed points are of form $X
= U_p \sqrt{-\Lambda_p} D$ and columns are scalar multiple of the $p$
low-lying eigenvectors of $A$. Global convergence aims at showing that the
iteration \eqref{eq:triofm} converges to one of the stable fixed points.
In order to simplify the notation, we define the set of stable fixed
points as $\calX^* = \left\{U_p \sqrt{-\Lambda_p}D \right\}$, where $U_p$,
$\Lambda_p$, and $D$ are defined in Theorem~\ref{thm:fixedpt}. Further,
the distance between a point $X$ and the set $\calX^*$ is denoted as
$\fnorm{X - \calX^*} = \min_{Y \in \calX^*} \fnorm{X-Y}$, which means the
smallest F-norm between $X$ and all points in $\calX^*$. For the first $i$
columns, we define the set of stable fixed points as $\calX^*_i =
\left\{U_i \sqrt{-\Lambda_i}D_i \right\}$, where $U_i$ is the first $i$
columns of $U_p$, $\Lambda_i$ and $D_i$ are the first $i$-by-$i$ principal
submatrices of $\Lambda_p$ and $D$ respectively.

Theorem~\ref{thm:global-conv} states that the iteration \eqref{eq:triofm}
converges to a global minimum almost surely. While the proof depends on
the following lemmas. Lemma~\ref{lem:bounded-domain} guarantees that
$X^{(t)}$ stays within the big domain as long as the initial point is in
there; Lemma~\ref{lem:single-global-conv} shows the global convergence of
$x_1$; Lemma~\ref{lem:multi-global-conv} shows the convergence of $x_i$
for $i > 1$. The proof of Lemma~\ref{lem:bounded-domain} is in
Appendix~\ref{app:whole_bound}, and the proofs of
Lemma~\ref{lem:single-global-conv} and Lemma~\ref{lem:multi-global-conv}
are in Section~\ref{sec:proof_details}.

\begin{lemma} \label{lem:bounded-domain}
    Assume Assumption~\ref{assump:init} is satisfied. For any iteration
    $t$, the iterate $X^{(t)}=(x_1^{(t)}, \dots, x_n^{(t)})$ satisfies
    $\norm{x_i^{(t)}} \leq R_i$ for all $1 \leq i \leq p$.
\end{lemma}

\begin{lemma} \label{lem:single-global-conv}
    Assume Assumption~\ref{assump:init} is satisfied and $x_1^{(0)}$
    is not perpendicular to $u_1$. Then $x_1^{(t)}$ converges to $\pm
    \sqrt{-\lambda_1} u_1$.
\end{lemma}

\begin{lemma} \label{lem:multi-global-conv}
    Assume Assumption~\ref{assump:init} is satisfied and $\lim_{t \rightarrow
    \infty} \fnorm{X_{i-1}^{(t)} - \calX^*_{i-1}} = 0$ for any $1\leq
    i\leq p$.  $x_i^{(t)}$ converges to one of the points in $\left\{ 0,
    \pm\sqrt{-\lambda_j}u_j \text{ for } j = i, \dots, p \right\}$.
\end{lemma}

\begin{theorem} \label{thm:global-conv}
    If Assumption~\ref{assump:init} is satisfied, then the iteration
    \eqref{eq:triofm} converges to $\calX^*$ for all initial points
    besides a set $W$ of measure zero.
\end{theorem}

\begin{proof}
This theorem is proved by induction. The set of these initial points is
denoted as $\calX_i$ for the first $i$ columns.
Lemma~\ref{lem:bounded-domain} guarantees that $\iota - \alpha g$ maps
points in $\calX_i$ to $\calX_i$, where $\iota$ denotes the identity
operator and $g$ is the operator defined in \eqref{eq:g}. We further
introduce a notation for unstable fixed points as $\calA_i^*$. Recall
Theorem~\ref{thm:fixedpt} for $p = i$, we can characterize $\calA_i^*$ as,
\begin{equation}
    \calA_i^* = \left\{ X \in \calX_i \middle| X =
    U_q \sqrt{-\Lambda_q} P S \text{ and } X \neq U_i
    \sqrt{-\Lambda_i} D \right\}.
\end{equation}

For the first column $x_1^{(t)}$, Lemma~\ref{lem:single-global-conv} shows
that $\lim_{t\rightarrow \infty}x_1^{(t)}=\pm \sqrt{-\lambda_1}u_1$ for
all $x_1^{(0)}$ not perpendicular to $u_1$. Alternatively, it can be
restated as $\lim_{t \rightarrow \infty} \fnorm{X_1^{(t)} - \calX_1^*} =
0$ for all initial points except those in $W_1 = \left\{X_1^{(0)} \middle|
u_1^\top X_1^{(0)} = 0\right\}= \Big\{X_1^{(0)} \Big| \lim_{t
\rightarrow \infty} (\iota - \alpha g)^t \left(X_1^{(0)}\right) \in
\calA_1^* \Big\}$. Obviously the set $W_1$ has measure zero.

Now we assume that the statement of Theorem~\ref{thm:global-conv} holds
for the first $i-1$ columns for $i \in (1, p]$, \ie, $\lim_{t \rightarrow
\infty} \fnorm{X_{i-1}^{(t)} - \calX_{i-1}^*} = 0$ for all initial points
except those in $W_{i-1}$ and the set $W_{i-1}$ has measure zero.

We first define the set $W_i$ for $i$ as,
\begin{equation}
    W_i = \left\{X_i^{(0)} \middle| \lim_{t \rightarrow \infty}
    (\iota - \alpha g)^t \left(X_i^{(0)}\right) \in \calA_i^*
    \right\} \bigcup V_i,
\end{equation}
for $V_i = \left\{X_i^{(0)} \middle| X_{i-1}^{(0)} \in W_{i-1}
\right\}$. Since $W_{i-1}$ has measure zero, we know that the
set $V_i$ also has measure zero. Next we focus on the points in
$\calX_i \setminus V_i$ and $W_i \setminus V_i$.

Here we apply Theorem~2 in \citet{Lee2019} to show that $W_i \setminus
V_i$ has measure zero. All conditions therein must be checked first. Since
the first $k-1$ columns are independent of the $k$-th one, the operator
$\iota - \alpha g$ is smooth and maps $\calX_i \setminus V_i$ to $\calX_i
\setminus V_i$. According to Theorem~\ref{thm:fixedpt}, points in
$\calA_i^*$ are unstable fixed points, so are points in $\calA_i^*
\setminus V_i$. The last thing to check is the invertibility of $\Diff
(\iota - \alpha g) = I - \alpha \Diff g$.  As has been discussed in the
proof of Theorem~\ref{thm:fixedpt} from \cite{Gao2020}, $\Diff g$ is a
block upper triangular matrix and its spectrum is determined by the
spectrum of the diagonal blocks $J_{kk}$ for $k = 1, \dots, i$,
\begin{equation} \label{eq:Jacobiansubmat}
    J_{kk} = A + X_k X_k^\top + x_k^\top x_k I + x_k x_k^\top.
\end{equation}
For points in $\calX_i$, the spectrum norm of $X_k X_k^\top$,
$x_k^\top x_k I$, and $x_k x_k^\top$ are upper bounded by $2R_p^2$,
$R_p^2$, and $R_p^2$ respectively. Further we have $\norm{A} <
\frac{R_p^2}{3}$. Hence, combined with the assumption on $\alpha$,
we have the following bound,
\begin{equation}
    \lambda (\Diff (\iota - \alpha g)) > \frac{1}{2}
\end{equation}
for all $X_i \in \calX_i$, which implies that $\mathrm{det} \left(\Diff
(\iota - \alpha g)\right) \neq 0$ for all $X_i \in \calX_i \setminus V_i$.
Therefore, Theorem~2 in \citet{Lee2019} can be applied, and the set $W_i
\setminus V_i$ has measure zero. Further, the set $W_i$ has measure zero.

Then Lemma~\ref{lem:multi-global-conv} implies that for $X_i^{(0)} \in
\calX_i \setminus W_i$ there is
\begin{equation}
    \lim_{t\rightarrow \infty}x_i^{(t)} = \pm \sqrt{-\lambda_i}u_i.
\end{equation}
Hence we have $\lim_{t \rightarrow \infty} \fnorm{X_i^{(t)} - \calX_i^*} =
0$ for all initial points except those in $W_i$ and the set $W_i$ has
measure zero. By induction, the theorem is proved.
\end{proof}

Theorem~\ref{thm:global-conv} shows the global convergence of \TOM{}
without rate. We do not expect any provable rate of the global convergence
since \TOM{} solves a non-convex problem and has unstable fixed points.

\section{Proofs of Lemmas}
\label{sec:proof_details}

This section proves Lemma~\ref{lem:single-global-conv} and
Lemma~\ref{lem:multi-global-conv}.

\subsection{Proof of Lemma~\ref{lem:single-global-conv}}

The iteration of the first column is a gradient descent method applied to
the single-column version of \eqref{eq:obj}.
Lemma~\ref{lem:single-global-conv} states that the iteration of the first
column converges globally. Combining the energy landscape analysis in
\cite{Gao2020, Li2019c} and the escaping saddle point analysis in
\cite{Lee2019}, the global convergence would be proved. In this section,
we provide another proof of the global convergence, where the convergence
rate is given implicitly. The proof is analog to that of the power method.
The convergence of angle is proved first and then the convergence of
vector length. In the following, $\theta^{(t)}$ denotes the acute angle
between $x_1^{(t)}$ and $\pm u_1$.

\begin{lemma} \label{app:lemma2}
    Assume Assumption~\ref{assump:init} is satisfied and $
    x_1^{(0)}$ is not perpendicular to $u_1$.  Then the tangent of
    $\theta^{(t)} = \angle (x_1^{(t)},u_1)$ linearly converges to 0,
    \ie, $\tan \theta^{(t+1)}\leq \frac{1 - \alpha \lambda_2} {1 -
    \alpha\lambda_1}\tan \theta^{(t)}$.
\end{lemma}
 
\begin{proof}
The eigenvectors of the symmetric matrix $A$ are orthonormal. Both the
2-norm and the angle $\theta^{(t)} = \angle (x_1^{(t)},u_1) = \angle
(U^\top x_1^{(t)},e_1)$ are invariant to orthogonal transform. Without
loss of generality, we assume that $A$ is a diagonal matrix with its
diagonal entries being $\lambda_1, \dots, \lambda_n$ and the corresponding
eigenvectors are $e_1, \dots, e_n$. Further, we drop the iteration index
in the superscript and denote the following iteration variables with
$\widetilde{\cdot}$. The first column of $X^{(t)}$ iterates as,
\begin{equation}
    \tx_{1} = \left(I - \alpha A - \alpha x_1^\top x_1 I\right)x_1.
\end{equation}
Let $x_{1k}$ and $\tx_{1k}$ denotes the $k$-th element of $x_1$ and
$\tx_1$ respectively. Then we have,
\begin{equation} \label{eq:elementiter}
    \tx_{1k} = \left(1 - \alpha \lambda_k - \alpha r_1^2 \right) x_{1k},
\end{equation}
where $r_1$ denotes the norm of $x_1$.

The tangent of $\ttheta = \angle (\tx_1, e_1)$ can be written in terms of
elements of $\tx_1$ and be bounded as,
\begin{align} \label{eq:thetaineq}
    \tan \ttheta & = \frac{ \sqrt{ \tx_{12}^2 + \tx_{13}^2 + \cdots +
    \tx_{1n}^2 } }{ \abs{\tx_{11}} } \notag\\
    & = \frac{ \sqrt{ \left(1 - \alpha \lambda_2 -
    \alpha r_1^2\right)^2 x_{12}^2 + \cdots + \left(1 - \alpha
    \lambda_n - \alpha r_1^2\right)^2 x_{1n}^2 } }{ \left(1 -
    \alpha \lambda_1 - \alpha r_1^2\right) \abs{x_{11}} } \notag\\
    & \leq \frac{1 - \alpha \lambda_2 - \alpha r_1^2}{1 - \alpha
    \lambda_1 - \alpha r_1^2} \left( \frac{ \sqrt{ x_{12}^2 +
    \cdots + x_{1n}^2 } }{ \abs{x_{11}} } \right)
    \leq \frac{1 -
    \alpha \lambda_2}{1 - \alpha \lambda_1} \tan \theta,
\end{align}
where the assumption on $\alpha$ guarantees the positivity of $1 - \alpha
\lambda_1 - \alpha r_1^2$ and $1 - \alpha \lambda_2 - \alpha r_2^2$.

Applying \eqref{eq:thetaineq} recursively, we prove the lemma.
\end{proof}
 
Lemma~\ref{app:lemma2} shows the linear convergence for the tangent of the
angle between $x_1^{(t)}$ and $u_1$. Next, we would focus on the
convergence of the vector length.
 
\begin{lemma} \label{app:lemma3}
    Assume Assumption~\ref{assump:init} is satisfied and $ x_1^{(0)}$
    is not perpendicular to $u_1$. Then there exists an integer $N$
    such that $\norm{x_1^{(t)}} \geq \frac{\sqrt{-2\lambda_q}}{4}$
    holds for all $t \geq N$.
\end{lemma}

\begin{proof}
Without loss of generality, we assume $A$ is diagonal as in the proof of
Lemma~\ref{app:lemma2} and the same notations are used here. We split
$x_1$ into two vectors as $x_1 = \begin{pmatrix} y_1^\top & y_2^\top
\end{pmatrix}^\top$ where $y_1 = \begin{pmatrix}x_{11} & \cdots &
x_{1q}\end{pmatrix}^\top$ and $y_2 = \begin{pmatrix}x_{1(q+1)} & \cdots &
x_{1n}\end{pmatrix}^\top$.

The proof consists of two parts. In the first part, we show that there
exist an iteration $N_1$, such that $\norm{y_1^{(N_1)}} \geq
\frac{\sqrt{-2\lambda_q}}{2}$ and $\norm{y_2^{(N_1)}} \leq
\frac{\sqrt{-2\lambda_q}}{4}$. In the second part, we show that as long as
the condition in the first part is satisfied, the length of $x_1$ will
never go below $\frac{\sqrt{-2\lambda_q}}{4}$.

Notice that \eqref{eq:elementiter} is bounded as,
\begin{equation} \label{eq:x1idecrease}
    \abs{\tx_{1k}} \leq \left(1 - \alpha \lambda_k - \alpha
    (x_{1k})^2 \right) \abs{x_{1k}}
\end{equation}
for $k = q+1, \dots, n$. All entries $\abs{x_{1k}}$ decays monotonically
to zero for $k > q$. Hence there exists an integer $M$ such that for any
$t \geq M$ we have $\norm{y_2^{(t)}} \leq \frac{\sqrt{-2\lambda_q}}{4}$.
Further, for $t \geq M$, if $\norm{y_1} \leq
\frac{\sqrt{-2\lambda_q}}{2}$, we have,
\begin{equation} \label{eq:x11increase}
    \abs{\tx_{1k}} \geq \left(1 - \alpha \lambda_k + \alpha
    \frac{5\lambda_q}{8} \right) \abs{x_{1k}},
\end{equation}
for $k = 1, \dots, q$, where the increasing factors are strictly greater
than one. Also we have $x_{11}^{(0)} = x_1^\top u_1 \neq 0$ in the
assumption. Considering the choice of $\alpha$,
$x_{11}^{(t)}$ remains nonzero throughout iterations.
Hence there exists a integer $N \geq M$ such that $\norm{y_1^{(N)}} \geq
\frac{\sqrt{-2\lambda_q}}{2}$ and $\norm{y_2^{(N)}} \leq
\frac{\sqrt{-2\lambda_q}}{4}$.

Next, for any $t \geq N$, if $\norm{y_1} \geq
\frac{\sqrt{-2\lambda_q}}{2}$, then we have,
\begin{equation}
    \norm{\ty_1} \geq (1 - \alpha r_1^2)
    \norm{y_1}
    \geq \left(1 - \frac{R_1^2}{5R_p^2} \right)
    \norm{y_1}
    \geq \frac{\sqrt{-2\lambda_q}}{4},
\end{equation}
where we adopt Lemma~\ref{lem:bounded-domain} and the assumption on
$\alpha$ in the second inequality. Such a relation means that as long as
$\norm{y_1} \geq \frac{\sqrt{-2\lambda_q}}{2}$, the length of the vector
in the next iteration is lower bounded by $\frac{\sqrt{-2\lambda_q}}{4}$.

When $\frac{\sqrt{-2\lambda_q}}{4} \leq \norm{y_1} \leq
\frac{\sqrt{-2\lambda_q}}{2}$ and $\norm{y_2} \leq
\frac{\sqrt{-2\lambda_q}}{4}$, we have,
\begin{equation}
    \norm{\ty_1} \geq (1 - \alpha \lambda_q - \alpha r_1^2)
    \norm{y_1} \geq \left(1 - \frac{3}{8} \alpha \lambda_q \right)
    \norm{y_1} \geq \frac{\sqrt{-2\lambda_q}}{4}.
\end{equation}

Hence, as long as $t \geq N$, we have $\norm{x_1^{(t)}} \geq
\norm{y_1^{(t)}} \geq \frac{\sqrt{-2\lambda_q}}{4}$.
\end{proof}

Now we are ready to prove Lemma~\ref{lem:single-global-conv} due to the
fact that angle always converges and the norm is lower-bounded away from
zero.

\begin{proof} ({\bf Proof of Lemma~\ref{lem:single-global-conv}}) Without loss of generality, we assume $A$ is diagonal as in the proof of
Lemma~\ref{app:lemma2} and the same notations are used here.

According to Lemma~\ref{app:lemma2}, the tangent of $\theta$ converges to
zero, \ie,
\begin{equation}
    \tan \theta = \frac{ \sqrt{ x_{12}^2 + x_{13}^2 + \cdots +
    x_{1n}^2 } }{ \abs{x_{11}} } \rightarrow 0.
\end{equation}
Lemma~\ref{lem:bounded-domain} implies the boundedness of
$x_1$, which implies the boundedness of $x_{11}$. Hence we have,
\begin{equation}
    \sqrt{ x_{12}^2 + x_{13}^2 + \cdots + x_{1n}^2 } \rightarrow 0.
\end{equation}
To simplify the notation, we denote $\eta$ as $\eta = \sqrt{x_{12}^2 +
x_{13}^2 + \cdots + x_{1n}^2 }$. The convergence of $\eta$ can be stated
as follows. For any $\veps \leq \min \left( \frac{\sqrt{-\lambda_q}}{4},
\frac{\sqrt{\lambda_1 \lambda_q}}{8R_1} \right)$, there exists an integer
$N_1$ such that for any $t \geq N_1$, we have $\eta^2 \leq \veps^2$. Also
recall Lemma~\ref{app:lemma3}, there exists an integer $N_2$, such that
for any $t \geq N_2$, we have $\norm{x_1^{(t)}} \geq
\frac{\sqrt{-2\lambda_q}}{4}$.  Combining the bounds on $\eta$ and
$\norm{x^{(t)}_1}$, we have,
\begin{equation}
    \left( x^{(t)}_{11} \right)^2 =
    \norm{x^{(t)}_1}^2 - \eta^2 \geq -\frac{\lambda_q}{8} - \veps^2 \geq
    -\frac{\lambda_q}{16}
\end{equation}
for $t \geq M = \max (N_1, N_2)$.

Since the assumption on stepsize $\alpha$ guarantees the positivity of
$(1-\alpha \lambda_1 - \alpha r_1^2)$, $\tx_{11} = (1 - \alpha \lambda_1 -
\alpha r_1^2)x_{11}$ remains the same sign as $x_{11}$ and the same as
$x_{11}^{(0)}$. We first discuss the scenario $x_{11}^{(0)} > 0$.

Let $\delta^{(t)} = x_{11}^{(t)} - \sqrt{-\lambda_1}$. We have the
relationship,
\begin{align}
    \delta^{(t+1)} & = x_{11}^{(t+1)} - \sqrt{-\lambda_1}
    = \left( 1 - \alpha \left( \lambda_1 +
    \left(x_{11}^{(t)}\right)^2 + \left(\eta^{(t)}\right)^2
    \right) \right) x_{11}^{(t)} - \sqrt{-\lambda_1} \notag\\
    & = \left( 1 - \alpha \left( \sqrt{-\lambda_1} +
    x_{11}^{(t)}\right) x_{11}^{(t)} \right) \delta^{(t)} -
    \alpha \left(\eta^{(t)}\right)^2 x_{11}^{(t)}.
\end{align}
Taking the absolute value of both side, we obtain the inequality,
\begin{align}
    \abs{\delta^{(t+1)}}
    \leq & \left( 1 - \alpha \left( \sqrt{-\lambda_1} +
    x_{11}^{(t)}\right) x_{11}^{(t)}
    \right) \abs{\delta^{(t)}} + \alpha
    \left(\eta^{(t)}\right)^2 x_{11}^{(t)} \notag\\
    \leq & \left( 1 - \alpha \frac{\sqrt{\lambda_1 \lambda_q}}{4}
    \right) \abs{\delta^{(t)}} + \alpha
    \veps^2 R_1 \notag\\
    \leq & \left( 1 - \alpha \frac{\sqrt{\lambda_1 \lambda_q}}{4}
    \right)^{t+1-M} \abs{\delta^{(M)}} + 
    \veps^2 \frac{4R_1}{\sqrt{\lambda_1
    \lambda_q}} \notag\\
    \leq & \left( 1 - \alpha \frac{\sqrt{\lambda_1 \lambda_q}}{4}
    \right)^{t+1-M} \abs{\delta^{(M)}} + 
    \frac{\veps}{2}. 
\end{align}
Hence there exists an integer $N \geq M$ such that for any $t \geq N$,
$\abs{\delta^{(t)}} \leq \veps$.

If $x_{11}^{(0)} < 0$, the iteration converges to $-\sqrt{-\lambda_1}$.
The analysis is analog to the above one. The lemma is proved.
\end{proof}

\subsection{Proofs of Lemma~\ref{lem:multi-global-conv}}

We give two proofs of Lemma~\ref{lem:multi-global-conv} in
Section~\ref{sec:proof1} and Section~\ref{sec:proof2} respectively. The
first proof is inspired by the noisy power method and follows closely as
that of Lemma~\ref{app:lemma2}. The second proof is related to the noisy
optimization method. We give the two proofs to hint at global convergence
proofs for other algorithms in \TOM{} family.

\subsubsection{Proof inspired by noisy power method}
\label{sec:proof1}

We now turn to the multicolumn case. When we are proving the multicolumn
case, we first assume the fact that all previous columns have converged to
global minima, \ie,
\begin{equation}
    \lim_{t \rightarrow \infty} \fnorm{ X_{i-1}^{(t)} - \calX^*_{i-1}
    } = 0.
\end{equation}
In the following, we first prove a few lemmas to support the proof of
Lemma~\ref{lem:multi-global-conv}.

\begin{lemma}\label{app:lemma5}
    Assume Assumption~\ref{assump:init} is satisfied and $\lim_{t
    \rightarrow \infty} \fnorm{ X_{i-1}^{(t)} - \calX^*_{i-1} } = 0$. Then
    $\lim_{t \rightarrow \infty} u_k^\top x_i^{(t)} = 0$ for all integer
    $k \in [1, i) \bigcup (q, n]$.
\end{lemma}

\begin{proof}
First we will introduce some notations. Let $E$ be the symmetric residual
matrix of the first $i-1$ columns, \ie, $E^{(t)} = \sum_{k = 1}^{i-1}
\left(x_k^{(t)} \left(x_k^{(t)}\right)^\top + \lambda_k u_k
u_k^\top\right)$, and $E_k^{(t)}$ denote the $k$-th column of $E^{(t)}$.
The convergence of $X^{(t)}_{i-1}$ implies that $\lim_{t \rightarrow
\infty} \fnorm{E^{(t)}} = 0$ and hence $\lim_{t \rightarrow \infty}
\norm{E_k^{(t)}} = 0$ for any $k = 1, 2, \dots, i-1$. Using the notation
$E^{(t)}$ and \eqref{eq:triofm}, the iteration for the $i$-th column of
$X^{(t)}$ can be written as,
\begin{equation} \label{eq:it-for-xi}
    x_i^{(t+1)} = \left(I - \alpha \tA - \alpha \left(x_i^
    {(t)}\right)^\top x_i^{(t)} I \right) x_i^{(t)}-\alpha
    E^{(t)} x_i^{(t)},
\end{equation}
where $\tA = A - \sum_{k=1}^{i-1} \lambda_k u_k u_k^\top$.

Without loss of generality, we assume $A$ is diagonal. In the following,
we consider the convergence of $x_{ik}^{(t)}$ for $k \in [1, i) \bigcup
(q, n]$. For any $\veps < \sqrt{\frac{2}{\alpha}}$, there exists $N$ such
that $\norm{E^{(t)}_k} < \frac{\veps^3}{2R_i}$ holds for all $t \geq N$.
Multiplying $e_k^\top$ on both sides of \eqref{eq:it-for-xi}, we have,
\begin{align}
    \abs{ x_{ik}^{(t+1)} } & = \abs{ \left( 1 - \alpha \left(
    x_i^{(t)} \right)^\top x_i^{(t)} \right)x_{ik}^{(t)} -
    \alpha e_k^\top \tA x_{i}^{(t)} -
    \alpha \left( E_k^{(t)} \right)^\top x_i^{(t)}} \notag\\
    & \leq \left(1 - \alpha \left( x_{ik}^{(t)} \right)^2
    \right) \abs{x_{ik}^{(t)}} + \alpha \norm{E_k^{(t)}}
    \norm{x_i^{(t)}}
    \leq \left( 1 - \alpha \left(x_{ik}^{(t)}\right)^2 +
    \frac{\alpha \veps^3}{2 \abs{ x_{ik}^{(t)}} }\right)
    \abs{ x_{ik}^{(t)} },
\end{align}
for $k \in [1, i) \bigcup (q, n]$, where we adopt the assumption on
$\alpha$ and Cauchy-Schwartz inequality.

When $\abs{x_{ik}^{(t)} } > \veps$, the inequality can be bounded as $\abs{
x_{ik}^{(t+1)} } \leq \left(1 - \frac{\alpha \veps^2}{2} \right) \abs{
x_{ik}^{(t)} }$, which means $\abs{ x_{ik}^{(t)} }$ decays exponentially
with the factor $1 - \frac{\alpha \veps^2}{2}$. On the other hand, if there
is a $t$ such that $\abs{ x_{ik}^{(t)} } \leq \veps$, then the quantity in
the following iteration is upper bounded by
\begin{equation}
    \abs{ x_{ik}^{(t+1)} } \leq \left(1+\alpha
    R_i^2\right)\veps + \frac{\alpha \veps^3}{2} \leq \left(
    2 + \alpha R_i^2 \right) \veps,
\end{equation}
where the second inequality holds due to $\veps <
\sqrt{\frac{2}{\alpha}}$.

Hence we conclude that, for any $\veps$, there exist a constant
$N^{\prime} > N$, such that $\abs{ x_{ik}^{(t)} } \leq \left( 2 + \alpha
R_i^2 \right) \veps$ holds for all $t \geq N^{\prime}$. Thus we have
$\lim_{t \rightarrow \infty} u_k^\top x_i^{(t)} = 0$ for all $k \in [1, i)
\bigcup (q, n]$.
\end{proof}

Lemma~\ref{app:lemma6} is the multicolumn version of
Lemma~\ref{app:lemma3}.

\begin{lemma}\label{app:lemma6}
    Assume Assumption~\ref{assump:init} is satisfied and $\lim_{t
    \rightarrow \infty} \fnorm{ X_{i-1}^{(t)} - \calX^*_{i-1}} = 0$. If
    there exists an integer $k \in [i, q]$ such that $u_i^\top
    x_{i}^{(t)}$ does not converge to zero, then there exists an integer
    $N$ such that $\norm{x_k^{(t)}}\geq \frac{\sqrt{-2\lambda_q}}{4}$
    holds for all $t > N$.
\end{lemma}

\begin{proof}
Without loss of generality, we assume $A$ is diagonal. Notations remain
the same as that in the proof of Lemma~\ref{app:lemma5} if not redefined.
We split the vector $x_i^{(t)}$ into three parts: $y_1^{(t)} =
\left(x_{i1}^{(t)}, \dots, x_{i(i-1)}^{(t)} \right)^\top$, $y_2^{(t)} =
\left(x_{ii}^{(t)}, \dots, x_{iq}^{(t)} \right)^\top$, and $y_3^{(t)} =
\left(x_{i(q+1)}^{(t)}, \dots, x_{in}^{(t)} \right)^\top$.

From the assumption there exists an integer $k \in [i,q]$ such that
$u_k^\top x_{i}^{(t)}$ does not converge to zero.  Hence, there exists a
positive $\veps_0 < \frac{\sqrt{-2\lambda_q}}{8}$, such that for any $N$
there exists a $t > N$ and $\norm{y_2^{(t)}} > \veps_0$ holds. Further, we
have the convergence of $E^{(t)}$ and Lemma~\ref{app:lemma5} guarantees
the convergence of $y_1^{(t)}$ and $y_3^{(t)}$.  Thus for such $\veps_0$,
there exists an $N_1$ such that $\norm{E^{(t)}} \leq
\frac{\sqrt{-2\lambda_q} \veps_0}{4}$, $\norm{y_1^{(t)}} <
\frac{\sqrt{-2\lambda_q}}{8}$, and $\norm{y_3^{(t)}} <
\frac{\sqrt{-2\lambda_q}}{8}$ hold for all $t \geq N_1$, and
$\norm{y_2^{(N_1)}} > \veps_0$. The $j$-th entry of $x_i^{(t)}$ for $j \in
[i,q]$ satisfies,
\begin{equation}
    x_{ij}^{(t+1)} = \left(1 - \alpha \lambda_j - \alpha
    \norm{x_i^{(t)}}^2 \right) x_{ij}^{(t)} - \alpha e_j^\top
    E^{(t)} x_i^{(t)}.
\end{equation}
If $\norm{x_i^{(t)}} \leq \frac{\sqrt{-2\lambda_q}}{2}$ for $t =
N_1$, then we can bound the norm of $y_2^{(t)}$ as,
\begin{align}
    \norm{y_2^{(t+1)}} & \geq \left(1-\alpha
    \lambda_q - \alpha \norm{x_i^{(t)}}^2 \right)
    \norm{y_2^{(t)}} -\alpha \norm{E^{(t)}} \norm{x_i^{(t)}} \notag\\
    & \geq \left(1 - \frac{\alpha \lambda_q}{2} - \alpha
    \frac{\norm{E^{(t)}} \norm{x_i^{(t)}}}{\norm{y_2^{(t)}}}\right)
    \norm{y_2^{(t)}}
    \geq \left(1-\frac{\alpha \lambda_q}{4} \right)
    \norm{y_2^{(t)}}.
\end{align}
The increasing factor is strictly greater than one. Hence
$\norm{y_2^{(t+1)}} > \veps_0$ holds for $t+1$ as well.  And
$\norm{y_2^{(t)}}$ increases monotonically until $\norm{x_i^{(t)}} >
\frac{\sqrt{-2\lambda_q}}{2}$.

When $\norm{x_i^{(t)}} > \frac{\sqrt{-2\lambda_q}}{2}$, the norm of the
following iteration is lower bounded as,
\begin{align}
        \norm{x_i^{(t+1)}} & \geq
        \left(1-\alpha \lambda_n-\alpha R_i^2\right)
        \norm{x_i^{(t)}} - \alpha
        \norm{E^{(t)}} R_i \notag\\
        & \geq \left(1-\alpha \lambda_n -\alpha
        R_i^2\right) \frac{\sqrt{-2\lambda_q}}{2} - \alpha
        \frac{\sqrt{-2\lambda_q} \veps_0}{4}
        \geq \frac{\sqrt{-2\lambda_q}}{4},
\end{align}
where the last inequality is due to the assumption on $\alpha$. Further,
the norm of $y_2^{(t+1)}$ can be lower bounded as,
\begin{equation}
    \norm{y_2^{(t+1)}} \geq 
    \sqrt{\norm{x_k^{(t+1)}}^2 - \norm{y_1^{(t+1)}}^2 -
    \norm{y_3^{(t+1)}}^2}
    \geq \sqrt{\frac{-2\lambda_q}{16} + 2\frac{2\lambda_q}{64}}
    > \veps_0.
\end{equation}

Therefore, the norm of $x_i^{(t)}$ is lower bounded by
$\frac{\sqrt{-2\lambda_q}}{4}$ after the first iteration later than $N_1$
such that $\norm{x_i^{(t)}} > \sqrt{\frac{-2\lambda_q}{2}}$.
\end{proof}

Lemma~\ref{app:lemma7} and Lemma~\ref{app:lemma8} serve as the multicolumn
version of Lemma~\ref{app:lemma2}.  More precisely, under the assumption
that $x_i^{(t)}$ does not converge to zero, Lemma~\ref{app:lemma7} and
Lemma~\ref{app:lemma8} prove that there exists a tangent of
$\theta_i^{(t)} = \angle (x_i^{(t)},u_i)$ or $\theta_j^{(t)}=\angle
(x_i^{(t)},u_j)$ for $j \in (i,q]$ converging linearly to zero, where as
before $\theta_j^{(t)}$ denotes the acute angle between $x_i^{(t)}$ and
$\pm u_j$ for $j \in [i, q]$.

\begin{lemma}\label{app:lemma7}
    Assume Assumption~\ref{assump:init} is satisfied and $\lim_{t
    \rightarrow \infty} \fnorm{ X_{i-1}^{(t)} - \calX^*_{i-1} } = 0$. If
    $u_i^\top x_{i}^{(t)}$ does not converge to zero, then the tangent of
    $\theta_i^{(i)} = \angle (x_i^{(t)},u_i)$ converges to 0.
\end{lemma}

\begin{proof}
Without loss of generality, we assume $A$ is diagonal. Notations remain
the same as that in the proof of Lemma~\ref{app:lemma5} if not redefined.
Based on the assumptions that $e_i^\top x_i^{(t)} = x_{ii}^{(t)}$ does not
converge to zero, there exists a positive number $\delta > 0$ such that
for any $N$, there exists a $t>N$ and $\abs{x_{ii}^{(t)}} > \delta$.  We
also know that Lemma~\ref{app:lemma6} holds and $\norm{E^{(t)}}$ converges
to zero. Hence, for any $\veps$ sufficiently small, there exists an
integer $N_1$ such that $\norm{E^{(t)}} \leq \veps^2$ and
$\norm{x_i^{(t)}} \geq \frac{\sqrt{-2\lambda_q}}{4}$ hold for all $t >
N_1$.  Since $e_i^\top x_i^{(t)} = x_{ii}^{(t)}$ does not converge to
zero, there exists an integer $N_2 > N_1$, such that,
\begin{equation} \label{eq:ratio-cond}
    \cos \theta_i^{(N_2)} =
    \frac{\abs{x_{ii}^{(N_2)}}}{\norm{x_{i}^{(N_2)}}}
    \geq \frac{\delta}{R_i} \geq
    \frac{ 2 \veps^2}{\lambda_{i+1}-\lambda_i},
\end{equation}
Recall the definition of the tangent of $\theta_i^{(t)}$,
\begin{equation} \label{eq:tan-theta}
    \tan \theta_i^{(t+1)} = \frac{\sqrt{ \sum_{j \neq i} \left(
    x_{ij}^{(t+1)} \right)^2 }}{\abs{x_{ii}^{(t+1)}}}.
\end{equation}
We derive the lower bound and the upper bound for the denominator
and numerator respectively when $t = N_2$.

Using the iterative relationship \eqref{eq:it-for-xi}, we have the
lower bound on the denominator,
\begin{align} \label{eq:denominator-lower-bound}
    \abs{x_{ii}^{(N_2+1)}} & = \abs{ \left(1-\alpha \lambda_i
    - \alpha \norm{x_i^{(N_2)}}^2 \right) x_{ii}^{(N_2)} -
    \alpha e_i^\top E^{(N_2)} x_i^{(N_2)}} \notag\\
    & \geq \left(1-\alpha \lambda_i - \alpha
    \norm{x_i^{(N_2)}}^2 - \alpha \frac{\norm{E^{(N_2)}}
    \norm{x_i^{(N_2)}} }{\abs{x_{ii}^{(N_2)}}} \right)
    \abs{x_{ii}^{(N_2)}} \notag\\
    & \geq \left(1 - \alpha \frac{\lambda_i + \lambda_{i+1}}{2}
    - \alpha \norm{x_i^{(N_2)}}^2 \right) \abs{x_{ii}^{(N_2)}},
\end{align}
where the second inequality is due to \eqref{eq:ratio-cond}.

Regarding the numerator in \eqref{eq:tan-theta}, again using the
iterative relationship \eqref{eq:it-for-xi}, we have,
\begin{align} \label{eq:numerator-upper-bound}
    \sqrt{ \sum_{j \neq i} \left( x_{ij}^{(N_2+1)} \right)^2 }
    & \leq \Bigg\{ \sum_{j \neq i} \bigg[ \left(1 - \alpha
    \lambda_{i+1} - \alpha \norm{x_i^{(N_2)}}^2\right)^2
    \left(x_{ij}^{(N_2)}\right)^2 \notag  \\
    & \quad + \alpha^2 \norm{E^{(N_2)}}^2
    \norm{x_i^{(N_2)}}^2 \notag\\
    & \quad + 2 \alpha \left(1 - \alpha \lambda_{i+1} - \alpha
    \norm{x_i^{(N_2)}}^2 \right) \abs{ x_{ij}^{(N_2)}
    } \norm{E^{(N_2)}} \norm{x_i^{(N_2)}} \bigg]
    \Bigg\}^{\frac{1}{2}} \notag\\
    & \leq \left(1 - \alpha \lambda_{i+1} - \alpha
    \norm{x_i^{(N_2)}}^2\right) \sqrt{ \sum_{j\neq i}
    \left(x_{ij}^{(N_2)}\right)^2} + 3 \sqrt{\alpha n}R_i
    \veps.
\end{align}
The first inequality adopts the fact that, without $i$-th entry,
$\lambda_{i+1}$ is the smallest eigenvalue of $\tA$; the second inequality
mainly uses the inequality of the square-root function; and the last
inequality holds for sufficiently small $\veps$.

Substituting \eqref{eq:denominator-lower-bound} and
\eqref{eq:numerator-upper-bound} into \eqref{eq:tan-theta},
we obtain,
\begin{align} \label{eq:tan-ineq}
    \tan \theta_i^{(N_2+1)} & \leq \frac{1 - \alpha
    \lambda_{i+1} - \alpha \norm{x_i^{(N_2)}}^2}{1 -
    \alpha \frac{\lambda_i+\lambda_{i+1}}{2} - \alpha
    \norm{x_i^{(N_2)}}^2} \tan \theta_i^{(N_2)} +
    \frac{3\sqrt{\alpha n}R_i^2 \veps}{\frac{1}{2}
    \norm{x_i^{(N_2)}} \delta} \notag\\
    & \leq (1 - \beta) \tan \theta_i^{(N_2)} - \beta \tan
    \theta_i^{(N_2)} + C \veps,
\end{align}
where $\beta = \frac{1}{2} \left(1 - \frac{1 - \alpha
\lambda_{i+1}}{1 - \alpha \frac{\lambda_i + \lambda_{i+1}}{2}}
\right) = \frac{\alpha (\lambda_{i+1} - \lambda_i)}{4 - 2\alpha
(\lambda_i + \lambda_{i+1})} \in (0,1)$ and $C = \frac{24\sqrt{\alpha
n}R_i^2}{\sqrt{-2\lambda_q} \delta}$.

Based on \eqref{eq:tan-ineq}, if $\tan \theta_i^{(N_2)} >
\frac{C\veps}{\beta}$, than we have $\tan \theta_i^{(N_2+1)} < (1 - \beta)
\tan \theta_i^{(N_2)}$, which implies $\cos \theta_i^{(N_2+1)} > \cos
\theta_i^{(N_2)}$ due to the fact that all angles are acute. Therefore,
\eqref{eq:ratio-cond} holds for $t = N_2+1$ and $\tan \theta_i^{(t)}$
decay monotonically until $\tan \theta_i^{(t)} \leq \frac{C
\veps}{\beta}$. When $\tan \theta_i^{(t)} \leq \frac{C \veps}{\beta}$, we
obviously have $\tan \theta_i^{(t+1)} \leq \frac{C\veps}{\beta}$. The
inequality condition \eqref{eq:tan-ineq} still holds as long as $\veps$ is
sufficiently small.  Hence there exists a $N$ such that for all $t > N$,
we have $\tan \theta_i^{(t+1)} \leq \frac{C\veps}{\beta}$, which can be
arbitrarily small.
\end{proof}

\begin{lemma}\label{app:lemma8}
    Assume Assumption~\ref{assump:init} is satisfied and $\lim_{t
    \rightarrow \infty} \fnorm{ X_{i-1}^{(t)} - \calX^*_{i-1} } = 0$. If
    $u_i^\top x_{i}^{(t)}$ converges to zero and there exists an integer
    $j \in (i, q]$ such that $u_j^\top x_i^{(t)}$ does not converge to
    zero, then there exists an integer $k \in (i, q]$ such that the
    tangent of $\theta_k^{(t)} = \angle (x_i^{(t)},u_k)$ converges to 0.
\end{lemma}

Lemma~\ref{app:lemma8} is fairly similar to Lemma~\ref{app:lemma7}. The
proof can be found in Appendix~\ref{app:lemma8-proof}. Now all ingredients
for proving Lemma~\ref{lem:multi-global-conv} are prepared. We then prove Lemma~\ref{lem:multi-global-conv}.

\begin{proof}({\bf Proof of Lemma~\ref{lem:multi-global-conv}}) Without loss of generality, we assume $A$ is diagonal. Notations remain
the same as that in the proof of Lemma~\ref{app:lemma5} if not redefined.
Lemma~\ref{app:lemma5} implies that under the given assumptions, we have
$\lim_{t \rightarrow \infty} u_k^\top x_i^{(t)} = 0$ for all integer $k
\in [1, i) \bigcup (q, n]$.

If $u_k^\top x_i^{(t)}$ converges to zero for all $k \in [i, q]$, then
$x_i^{(t)}$ converges to zero vector, which is included in the statement
of Lemma~\ref{lem:multi-global-conv}.

Otherwise, there exists an integer $j \in [i, q]$ such that $u_j^\top
x_i^{(t)}$ does not converge to zero.  Hence the condition in
Lemma~\ref{app:lemma6} is satisfied. At the same time, either of
Lemma~\ref{app:lemma7} or Lemma~\ref{app:lemma8} holds, which means that
there exists an integer $k \in [i, q]$ such that the tangent of
$\theta_k^{(t)}$ converges to zero.

Now we focus on the convergence of $x_{ik}^{(t)}$. We denote $\eta$
as $\eta = (x_{i1}^2 + \cdots + x_{i(k-1)}^2 + x_{i(k+1)}^2 +
\cdots + x_{in}^2)^{\frac{1}{2}}$. Notice that $\sin \theta_k^{(t)} =
\frac{\eta^{(t)}}{\norm{x_i^{(t)}}}$ converges to zero due to the convergence
of the tangent and boundedness of the cosine function.  Lemma~\ref{app:lemma6}
also shows the $\norm{x_i^{(t)}}$ is lower bounded by a constant for $t$
large. Hence we conclude that $\eta^{(t)}$ converges to zero.

Due to the convergence of $\eta^{(t)}$ and Lemma~\ref{app:lemma6},
for any $\veps \leq \min \left(\frac{\sqrt{-\lambda_q}}{4},
\frac{\sqrt{\lambda_i \lambda_q}}{8R_i} \right)$, there exists an integer
$M$ such that for any $t \geq M$, we have $\eta^{(t)} \leq \veps$,
$\norm{x_i^{(t)}} \geq \frac{\sqrt{-2\lambda_q}}{4}$, and $\norm{E^{(t)}}
\leq \veps^2$. Combining these bounds, we obtain, $\left( x^{(t)}_{ik}
\right)^2 = \norm{x^{(t)}_i}^2 - \eta^2 \geq -\frac{\lambda_q}{8} -
\veps^2 \geq -\frac{\lambda_q}{16}$.

Since the stepsize $\alpha$ is small, the signs of $x_{ik}^{(t)}$ and
$x_{ik}^{(0)}$ remain the same. We first discuss the scenario
$x_{ik}^{(0)} > 0$. Let $\delta^{(t)} = x_{ik}^{(t)} - \sqrt{-\lambda_k}$.
We have,
\begin{align}
    \abs{\delta^{(t+1)}} & \leq \left( 1 - \alpha
    \left( \sqrt{-\lambda_k} + x_{ik}^{(t)}\right)
    x_{ik}^{(t)} \right) \abs{\delta^{(t)}} + \alpha
    \left(\eta^{(t)}\right)^2 x_{ik}^{(t)} + \alpha
    \norm{E^{(t)}} \norm{x_i^{(t)}} \notag\\
    & \leq \left( 1 - \alpha \frac{\sqrt{\lambda_k \lambda_q}}{4}
    \right) \abs{\delta^{(t)}} + 2 \alpha
    \veps^2 R_i \notag\\
    & \leq \left( 1 - \alpha \frac{\sqrt{\lambda_k \lambda_q}}{4}
    \right)^{t+1-M} \abs{\delta^{(M)}} + 
    \frac{\veps}{2}.
\end{align}
for $t > M$. Hence there exists an integer $N \geq M$ such that for any $t
\geq N$, $\abs{\delta^{(t)}} \leq \veps$.

If $x_{ik}^{(0)} < 0$, the iteration converges to $-\sqrt{-\lambda_k}$.
The analysis is analog to the above one.
\end{proof}

\subsubsection{Proof inspired by noisy optimization method}
\label{sec:proof2}

In this section, we prove Lemma~\ref{lem:multi-global-conv} based on an
idea inspired by the noisy optimization method. As we mentioned earlier,
\TOM{} is not a gradient descent method of an energy function. While we
could show the convergence of \TOM{} under the measurement of energy
functions. In noisy (stochastic) optimization methods, the convergence is
proved with an adequately scaled bound on the noises. In the following, we
focus on the convergence of a given column and view the perturbation from
earlier columns as the ``noise'' on the energy function.

More precisely, we first derive the energy functions used in the proof.
Lemma~\ref{lem:multi-global-conv} guarantees the first $i-1$ columns
converge to one of the global minima $X_i^* \in \calX_i^* = \{U_p
\sqrt{-\Lambda_i} D\}$. Then for any $\veps_{i-1}<\sqrt{\rho}$, there
exists a step $T$ such that
$\norm{X_{i-1}^{(t)}-X_{i-1}^{*}}<\veps_{i-1}$, which is referred as the
error of previous columns in the rest paper, for any $t \geq T$. Further,
let $E^{(t)} = X_{i-1}^{(t)} \left(X_{i-1}^{(t)} \right)^\top -
X_{i-1}^{*} \left(X_{i-1}^{*}\right)^\top$. We could have a simple bound
on $\norm{E^{(t)}}$,
\begin{align}
    \norm{E^{(t)}} & = \left\|\left(X_{i-1}^{(t)} - X_{i-1}^{*}\right)
    \left(X_{i-1}^{*}\right)^\top
    + X_{i-1}^{*}\left(X_{i-1}^{(t)} - X_{i-1}^{*}\right)^\top
    \right. \notag \\
    & \quad \left. + \left(X_{i-1}^{(t)} 
    - X_{i-1}^{*}\right)\left(X_{i-1}^{(t)} - X_{i-1}^{*}\right)^\top
    \right\| \notag\\
    & \leq 2\norm{X_{i-1}^{*}} \veps_{i-1}+\veps_{i-1}^2
    \leq 3\sqrt{\rho} \veps_{i-1}.
\end{align}
Using notation $E^{(t)}$, we could rewrite the $i$-th column of
$g\left(X^{(t)}\right)$ as,
\begin{equation}
    g_i\left(X^{(t)}\right) = \tA x_i^{(t)} + x_i^{(t)} \left( x_i^{(t)} \right)^\top
    x_i^{(t)} + E^{(t)} x_i^{(t)},
\end{equation}
where $\tA = A + X_{i-1}^{*} \left(X_{i-1}^{*}\right)^\top$. Define an
energy function associated with $\tA$ as,
\begin{equation}
    F(x) = \fnorm{\tA + x x^\top}^2,
\end{equation}
where $F$ is index $i$ dependent since $\tA$ is $i$ dependent. Obviously,
$g_i\left(X^{(t)}\right)$ is the gradient of $F\left(x_i^{(t)}\right)$
with an extra term,
\begin{equation}
    g_i\left(X^{(t)}\right) = \nabla F\left(x_i^{(t)}\right) + E^{(t)} x_i^{(t)}.
\end{equation}
The Hessian of $F(x)$ is denoted as $H$, whose norm is upper bounded as,
\begin{equation}
    \norm{H\left(x_i\right)}=\norm{\tilde{A}+2x_ix_i^\top+ \left(x_i^\top x_i\right)I}\leq 4R_i^2.
\end{equation}
Now we will pave the path to proving Lemma~\ref{lem:multi-global-conv}.
First, we show that given the stepsize is sufficiently small, $x_i$ can
get close to the stationary point of $F$ after a finite number of
iterations.

\begin{lemma}\label{lem:en_decay}
    Assume Assumption~\ref{assump:init} is satisfied. If, for any small
    $\veps_i$, the error of previous columns satisfies $\veps_{i-1} < \min
    \{\sqrt{\rho}, \frac{\veps_i}{8\sqrt{3}R_i^2}\}$ and
    $\norm{\grad{F\left(x_i^{(t)}\right)}} > \veps_i$, then
    $F\left(x_i^{(t)}\right) - F\left(x_i^{(t+1)}\right) >
    \frac{1}{2}\norm{x_i^{(t)} - x_i^{(t+1)}} \veps_i > \frac{1}{4}\alpha
    \veps_i^2$.
\end{lemma}

\begin{proof}
In this proof, we omit the iteration index $t$, denote $x_i^{(t)}$ as
$x_i$, and denote $x_i^{(t+1)}$ as $\tx_i$. The Taylor expansion of $F(x)$
up to the second order admits,
\begin{equation}
    F\left(x_i\right) - F\left(x_i - \alpha g_i\right) =
    \alpha\grad{F\left(x_i\right)}^\top g_i
    - \frac{\alpha^2}{2} g_i^\top H(\xi) g_i,
\end{equation}
where $H$ is the hessian matrix of $F$ and $\xi$ is a point between $x_i$
and $\tx_i$. Recall the boundness of $\norm{E}$. The first order term can
be bounded as,
\begin{equation}
    \grad{F}^\top g_i \geq \norm{g_i}^2 - \norm{g_i}\norm{Ex_i}
    \geq \frac{1}{4} \norm{g_i}^2 + \frac{3}{4} \norm{g_i}\veps_i
    - \sqrt{3} R_i^2 \norm{g_i} \veps_{i-1},
\end{equation}
where $\grad{F} = \grad{F}(x_i)$ and the last inequality adopts the
assumption $\rho = \norm{A} <\frac{R_i^2}{3}$. Similarly, the second order
term can be bounded as
\begin{equation}
    g_i^\top H(\xi) g_i \leq 4R_i^2 \norm{g_i}^2.
\end{equation}
Combining two bounds together, we have,
\begin{align}
        F(x_i) - F(\tx_i) & >
        \frac{3\alpha}{4} \norm{g_i} \veps_i - \sqrt{3}\alpha R_i^2
        \norm{g_i}\veps_{i-1} +\frac{\alpha}{4} \norm{g_i}^2
        - 2 \alpha^2 R_i^2 \norm{g_i}^2 \notag \\
        & > \frac{1}{2} \norm{x_i - \tx_i} \veps_i
        > \frac{1}{4}\alpha\veps_i^2,
\end{align}
where all inequalities adopts the assumption on $\veps_{i-1}$ and $\alpha$.
\end{proof}

The function $F$ is bounded from below and Lemma~\ref{lem:en_decay} shows
that for any fixed $\veps_i$, the decreases of function values are also
bounded from below. Hence, after a finite number of iterations,
$x_i^{(t)}$ will always get close enough to one of the stationary points,
\ie, $\norm{F\left(x_i^{(t)}\right)}<\veps_i$. In the next lemma, we show
that $\norm{F\left(x_i^{(t)}\right)} < \veps_i$ defines a small neighbor
for each stationary point of $\grad{F(x)}$. Let $\calS$ denote the set of
all stationary points of $\grad{F(x)}$, \ie, $\grad{F(x)=0}$ for all $x
\in \calS$.

\begin{lemma}\label{lem:saddle_neighbor}
    Assume $\veps_i < \left( \min \left\{ \frac{\min_{i\leq k < q}
    \lambda_{k+1}-\lambda_k}{2}, \frac{-\lambda_q}{n}
    \right\}\right)^{\frac{3}{2}}$. If $x_i$ satisfies
    $\norm{\grad{F(x_i)}}<\veps_i$, then there exists $x_i^\prime \in
    \calS$ such that $\norm{x_i-x_i^\prime} \leq
    \sqrt{n}\veps_i^{\frac{1}{3}}$.
\end{lemma}

\begin{proof}
Without loss of generality, we assume $A$ is diagonal.
Then, we have $\calS = \{0,\pm
\sqrt{-\lambda_j}e_j, i \leq j \leq q\}$ and,
\begin{equation}
    \grad{F(x_i)} = \tA x_i + x_i x_i^\top x_i =
    \begin{bmatrix}
        r^2 x_{i1} & \cdots & r^2 x_{i(i-1)}
        & (r^2+\lambda_i)x_{ii} & \cdots
        & (r^2+\lambda_n)x_{in}
    \end{bmatrix}^\top,
\end{equation}
where $r=\norm{x_i}$. Now we consider a $x_i$ such that
$\norm{\grad{F(x_i)}} < \veps_i$. If $r^2 < \veps_i^{\frac{2}{3}}$, then
the lemma is proved since $x_i$ is close to $0 \in \calS$. Hence in the
rest of the proof, we assume $r^2 \geq \veps_i^{\frac{2}{3}}$.

We know the absolute values of entries in $\grad{F(x_i)}$ are less than
$\veps_i$. For the first $i-1$ entries, the inequalities $\abs{r^2 x_{ij}}
< \veps_i$ implies that $\abs{x_{ij}} < \veps_i^{\frac{1}{3}}$ for all $1
\leq j < i$. For other entries, the inequalities
$\abs{\left(r^2+\lambda_j\right) x_{ij}} < \veps_i$ implies either 
\begin{equation} \label{eq:case2}
    \abs{x_{ij}} < \veps_i^{\frac{1}{3}}
\end{equation}
or
\begin{equation} \label{eq:case1}
    \abs{r^2+\lambda_j} < \veps_i^{\frac{2}{3}}.
\end{equation}
If \eqref{eq:case2} holds for all $j \geq i$, we immediately have
$\norm{x_i} < \sqrt{n}\veps_i^{\frac{1}{3}}$.

Now, we consider the case \eqref{eq:case1} holds for some $j \geq i$. If
\eqref{eq:case1} holds for a $j > q$, then we have $r^2 <
\abs{r^2+\lambda_j} < \veps_i^{\frac{2}{3}}$ and the lemma is proved.
Further, the assumption on $\veps_i$ guarantees that there is at most a
single $j_0 \in [i, q]$ such that \eqref{eq:case1} holds. Under such
circumstance, we have,
\begin{equation}
    \abs{\sum_{j\neq j_0} x_{ij}^2 + x_{ij_0}^2 + \lambda_{j_0}} <
    \veps_i^{\frac{2}{3}}
    \quad \text{and} \quad
    \sum_{j\neq j_0}x_{ij}^2<(n-1)\veps_i^{\frac{2}{3}}.
\end{equation}
Simplifying these inequalities leads to,
\begin{equation}
    x_{ij_0} \in \left( -\sqrt{-\lambda_{j_0} + \veps_i^{\frac{2}{3}}}, 
    -\sqrt{-\lambda_{j_0} - n\veps_i^{\frac{2}{3}}} \right)
    \bigcup \left( \sqrt{-\lambda_{j_0} - n\veps_i^{\frac{2}{3}}},
    \sqrt{-\lambda_{j_0} + \veps_i^{\frac{2}{3}}} \right),
\end{equation}
where the assumption $\veps_i < (\frac{-\lambda_{j_0}}{n})^{\frac{3}{2}}$
guarantees the positivity of quantities under the radical sign. Using
inequalities $\sqrt{1-x} < 1-\frac{1}{2}x$ and $\sqrt{1-x} > 1-x$ for $x
\in (0,1)$, we further simplify the bounds on $x_{ij_0}$,
\begin{align}
    &\sqrt{-\lambda_{j_0}} - \frac{n}{\sqrt{-\lambda_{j_0}}}
    \veps_i^{\frac{2}{3}} < x_{ij_0}
    < \sqrt{-\lambda_{j_0}} + \frac{1}{2\sqrt{-\lambda_{j_0}}}
    \veps_i^{\frac{2}{3}}, \notag\\
    & \text{or} \notag\\
    &-\sqrt{-\lambda_{j_0}} - \frac{1}{2\sqrt{-\lambda_{j_0}}}
    \veps_i^{\frac{2}{3}} < x_{ij_0} < -\sqrt{-\lambda_{j_0}}
    - \frac{n}{\sqrt{-\lambda_{j_0}}} \veps_i^{\frac{2}{3}},
\end{align}
where the assumption on $\veps_i$ is used again. Finally,
$\veps^{\frac{2}{3}}$ is bounded by $\veps^{\frac{1}{3}}$ and the lemma is
proved.
\end{proof}

From Lemma~\ref{lem:en_decay} and Lemma~\ref{lem:saddle_neighbor}, we know
that after a finite number of iterations, $x_i^{(t)}$ will converge into a
neighborhood of a stationary point $s \in \calS$ with radius $\sqrt{n}
\veps_i^{\frac{1}{3}}$,
\begin{equation}
    \calN_{s,\veps_i} = \left\{ x \middle| \norm{\grad{F(x)}} < \veps_i
    \text{ and } \norm{x-s} < \sqrt{n} \veps_i^{\frac{1}{3}}
    \right\}.
\end{equation}
We also define the neighborhoods of other stationary points with
greater
or equal energy function values,
\begin{equation}
    \calM_{s,\veps_i} = \left\{ x \middle| \norm{\grad{F(x)}} < \veps_i
    \text{ and } \norm{x-s^\prime} < \sqrt{n} \veps_i^{\frac{1}{3}}
    \text{ for } s^\prime \in \calS, s^\prime \neq s, F(s^\prime) \geq F(s)
    \right\}.
\end{equation}
Now we discuss the case when $x_i^{(t)}$ leaves the neighborhood.

\begin{lemma} \label{lem:op-saddle}
    Assume $\veps_i < \left(\frac{\sqrt{-\lambda_j}}{9\sqrt{n}}\right)^3,
    \veps_{i-1} < \frac{\veps_i}{8\sqrt{3}R_i^2}$. If
    $x_i^{(t)} \in \calN_{s, \veps_i}$ for $s \in \calS$, then for any
    $t^\prime>t$, $x_i^{(t^\prime)} \not\in \calM_{s, \veps_i}$.
\end{lemma}

\begin{proof}
We prove this lemma by contradiction. Without loss of generality, we
assume $s = \sqrt{-\lambda_j}v_j$ for $j \geq i$. To simplify notations,
we drop the subscripts in $\calN_{s,\veps_i}$ and $\calM_{s,\veps_i}$.

Let $t_1 > t$ be the first time such that $x_i^{(t_1)} \not\in \calN$ and
$t_2 > t_1$ be the first time such that $x_i^{(t_2)} \in \calM$. More
precisely, we assume $x_i^{(t_2)}$ is in the neighborhood of $s^\prime$
for $F(s^\prime) \geq F(s)$ and the norm of the difference is lower
bounded as,
\begin{equation}
    \norm{s-s^\prime} \geq \min \left\{ \sqrt{-\lambda_j},
    \min_{j < k \leq q} \sqrt{-\lambda_j -\lambda_k} \right\}
    > 9 \sqrt{n} \veps_i^{\frac{1}{3}}.
\end{equation}
We also have,
\begin{align}
    \norm{\grad{F\left(x_i^{(t_1-1)}\right)}} < \veps_i, \quad
    \norm{x_i^{(t_1-1)} - s } < \sqrt{n} \veps_i^{\frac{1}{3}}, \quad
    \norm{\grad{F\left(x_i^{(t_1)}\right)}} \geq \veps_i, \notag \\
    \norm{\grad{F\left(x_i^{(t_2)}\right)}} < \veps_i, \quad
    \norm{x_i^{(t_2)} - s^\prime} < \sqrt{n}\veps_i^{\frac{1}{3}}.
\end{align}

In the following, we give two estimations on the energy difference
$F\left(x_i^{(t_1)}\right) - F\left(x_i^{(t_2)}\right)$ and derive the
contradiction. An upper bound on $F\left(x_i^{(t_1)}\right)$ admits,
\begin{align}\label{eq:est_leave}
    F\left(x_i^{(t_1)}\right) & = F\left(x_i^{(t_1-1)}\right)
    + \left(F\left(x_i^{(t_1)}\right)
    - F\left(x_i^{(t_1-1)}\right)\right)\notag \\
    & \leq F(s) + \sqrt{n}\veps_i^{\frac{4}{3}} +
    2\veps_i^2
    \leq F(s) + 2\sqrt{n}\veps_i^{\frac{4}{3}},
\end{align}
where the single step function difference can be upper bouneded in the
similar way as that in Lemma~\ref{lem:en_decay}. Combining
\eqref{eq:est_leave} with the lower bound of $F\left(x_i^{(t_2)}\right)$,
$F\left(x_i^{(t_2)}\right) \geq F\left(s^\prime\right) -
\sqrt{n}\veps_i^{\frac{4}{3}}$, we obtain an upper bound on the function
difference,
\begin{equation} \label{eq:paradox1}
    F\left(x_i^{(t_1)}\right) - F\left(x_i^{(t_2)}\right)
    \leq 3\sqrt{n}\veps_i^{\frac{4}{3}}.
\end{equation}

On the other hand, by Lemma~\ref{lem:en_decay}, the energy function keeps
decreasing for $x_i^{(t)}, t_1\leq t\leq t_2$.
\begin{align}
    F\left(x_i^{(t_1)}\right) - F\left(x_i^{(t_2)}\right)
    & = \sum_{t=t_1}^{t_2-1}
    \left(F\left(x_i^{(t)}\right) - F\left(x_i^{(t+1)}\right)\right) \notag\\
    & > \frac{1}{2} \veps_i \sum_{t=t_1}^{t_2-1}
    \left(\norm{x_i^{(t)}-x_i^{(t+1)}}\right)
    \geq \frac{1}{2}\veps_i\norm{x_i^{(t_1)}-x_i^{(t_2)}},
\end{align}
where the first inequality adopts Lemma~\ref{lem:en_decay}, and the
distance between $x_i^{(t_1)}$ and $x_i^{(t_2)}$ can be lower bounded as,
\begin{align}
    \norm{x_i^{(t_1)}-x_i^{(t_2)}} & \geq
    \norm{x_i^{(t_1-1)}-x_i^{(t_2)}}
    - \norm{x_i^{(t_1-1)}-x_i^{(t_1)}} \notag \\
    & \geq \norm{s-s^\prime} - 2 \sqrt{n} \veps_i^{\frac{1}{3}}
    - \alpha (\veps_i+3\sqrt{\rho}R_i \veps_{i-1})
    > 7 \sqrt{n} \veps_i^{\frac{1}{3}},
\end{align}
which contradict to \eqref{eq:paradox1}.
\end{proof}

\begin{lemma}\label{lem:same-saddle}
    Assume $\veps_i < \left(\frac{\sqrt{-\lambda_j}}{6\sqrt{n}}\right)^3,
    \veps_{i-1} < \frac{\veps_i}{8\sqrt{3}R_i^2}$. For any $t_1$, $t_2$
    with $t_1 < t_2$, and $s \in \calS$ such that $x_i^{(t_1)} \in
    \calN_{s,\veps_i}$, and $x_i^{(t_2)} \in \calN_{s, \veps_i}$, there is
    $\norm{x_i^{(t)} - s} < 6\sqrt{n} \veps_i^{\frac{1}{3}}$ for all $t
    \in [t_1, t_2]$.
\end{lemma}

Lemma~\ref{lem:op-saddle} shows that starting from a neighborhood of a
stationary point, the iteration will not converge to the neighborhoods of
different stationary points with greater or equal function values.
Lemma~\ref{lem:same-saddle} shows that if the iteration returns to the
same stationary point $s \in \calS$, all middle iterations are within a
neighborhood of $s$. The proof of Lemma~\ref{lem:same-saddle} adopts a
similar idea as that of Lemma~\ref{lem:op-saddle}. Please refer to
Appendix~\ref{app:same-saddle-proof} for the detail.

\begin{proof}({\bf Proof of Lemma~\ref{lem:multi-global-conv})} First, according to Lemma~\ref{lem:en_decay} and
Lemma~\ref{lem:saddle_neighbor}, for any $\veps_i$, there exists a time
$t$ such that $\norm{\grad{F\left(x_i^{(t)}\right)}} < \veps_i$ and hence
$\norm{x_i^{(t)}-s} \leq \sqrt{n} \veps_i^{\frac{1}{3}}$ for some $s \in
\calS$. This means that $x_i^{(t)} \in \calN_{s,\veps_i}$.

We then have a few scenarios.
\begin{enumerate}[(1)]
    \item The iteration stays in $\calN_{s, \veps_i}$ forever.
    \item The iteration leaves $\calN_{s, \veps_i}$ and returns to
    $\calN_{s, \veps_i}$ again. By Lemma~\ref{lem:same-saddle}, the
    iteration stays within an $\veps_i$ dependent neighborhood of $s$.
    \item The iteration leaves $\calN_{s, \veps_i}$ and enters another
    $\calN_{s^\prime, \veps_i}$ for $s^\prime \in \calS$ and $s^\prime
    \neq s$. By Lemma~\ref{lem:op-saddle}, we have $F(s^\prime) < F(s)$.
\end{enumerate}
Since $F(x)$ only has a finite number of stationary points, the third
scenario happens a finite number of times. Hence there exists a $T$ such
that for all $t \geq T$, $x_i^{(t)}$ stays within an $\veps_i$
neighborhood of a stationary point. Since $\veps_i$ is a free parameter
and can go to zero, we proved the lemma.
\end{proof}

\begin{remark}

We notice that the proof given in Section~\ref{sec:proof2} could be
extended to \TOM{} applied to other functions. Lemma~\ref{lem:en_decay},
Lemma~\ref{lem:op-saddle}, and Lemma~\ref{lem:same-saddle} are very much
objective function independent. Hence the extensions are straightforward.
The only part which shall be treated carefully is
Lemma~\ref{lem:saddle_neighbor}. The idea, `if $\norm{\grad{F(x)}}$ is
small, then $x$ is close to a stationary point', seems to be natural but
requires some work. Overall, we believe that our proof of convergence is
transferable.

\end{remark}

\section{Numerical Results}
\label{sec:num_res}

In the previous sections, we analyzed the theoretical converging
behavior of \TOM{}. In this section, we will show a numerical
result to support our analysis. A gradient descent algorithm with a
fixed stepsize is adopted, which is in accordance with the analysis
above. Though requiring plenty of iterations to converge, the numerical
results here mainly serve to depict the steady convergence behavior
and validate our analysis result. Numerical acceleration techniques
are proposed and explored in our companion paper~\cite{Gao2020}.

We compute the low-lying eigenpairs of a two-dimensional Hubbard model
under the FCI framework, which is defined on a lattice of size $3\times 3$
with 8 electrons (4 spin-up and 4 spin-down). The matrix size is $n\approx
1.7\times 10^3$. We adopt two expressions to moniter the convergence: one
for the objective function value $e_{\text{obj}}=f(X^{(t)})-f(X^\star)$, and
another for the accuracy of eigenvectors,
\begin{equation} \label{eq:evec}
    e_{\text{vec}} = \min_{ X^* \in
    \calX^*} \frac{\fnorm{X - X^*}}{\fnorm{X^*}},
\end{equation}
where in both expressions $\calX^*$ denotes the set of all stable fixed
points of the \TOM{}. We stop the iteration with a tolerance being $tol =
10^{-6}$. Though this is not of high accuracy from numerical linear
algebra point of view. As we have tested and proved~\cite{Gao2020}, the
method would converges steadily to the global minima with a linear rate in
local neighborhoods of global minima.

As we have proved in Section~\ref{sec:globalconv}, \TOM{} avoids the
saddle points and converges to the global minima with probability one. In
order to demonstrate the ability to escape from saddle points, we
initialize the iteration near a saddle point. Each time, we first randomly
choose a saddle point and randomly perturb it to be the initial point,
\ie, $\norm{X^{(0)} - X_{\text{saddle}}} < 10^{-6}$.

In all more than 100 experiments we have tested, $X^{(t)}$ always escape
from saddle points and converge to $\calX^*$. Here we depict two typical
convergence behaviors in Figure~\ref{fig:conv_behave}.

\begin{figure}[htp]
    \includegraphics[width=0.48\textwidth]{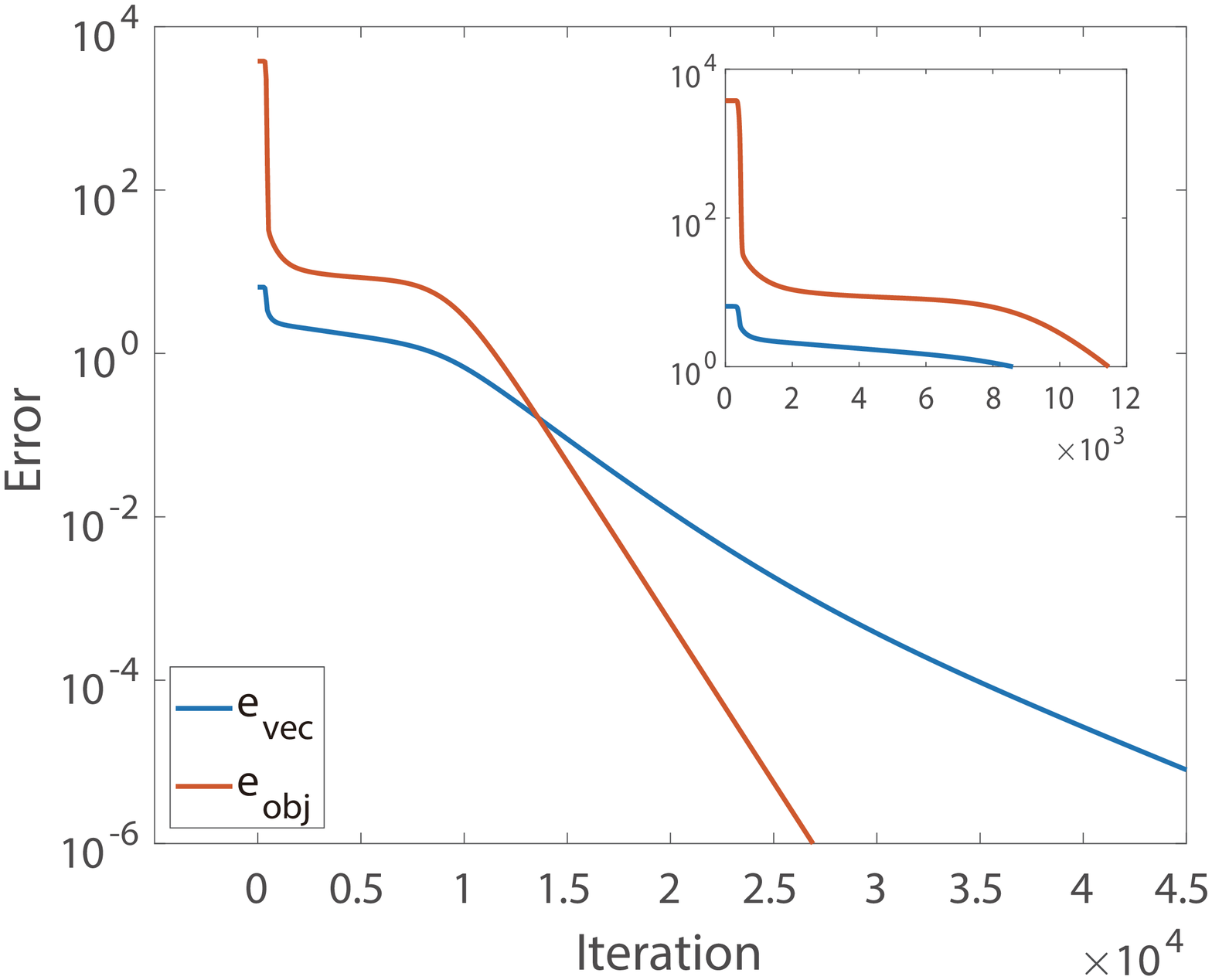}
    \includegraphics[width=0.48\textwidth]{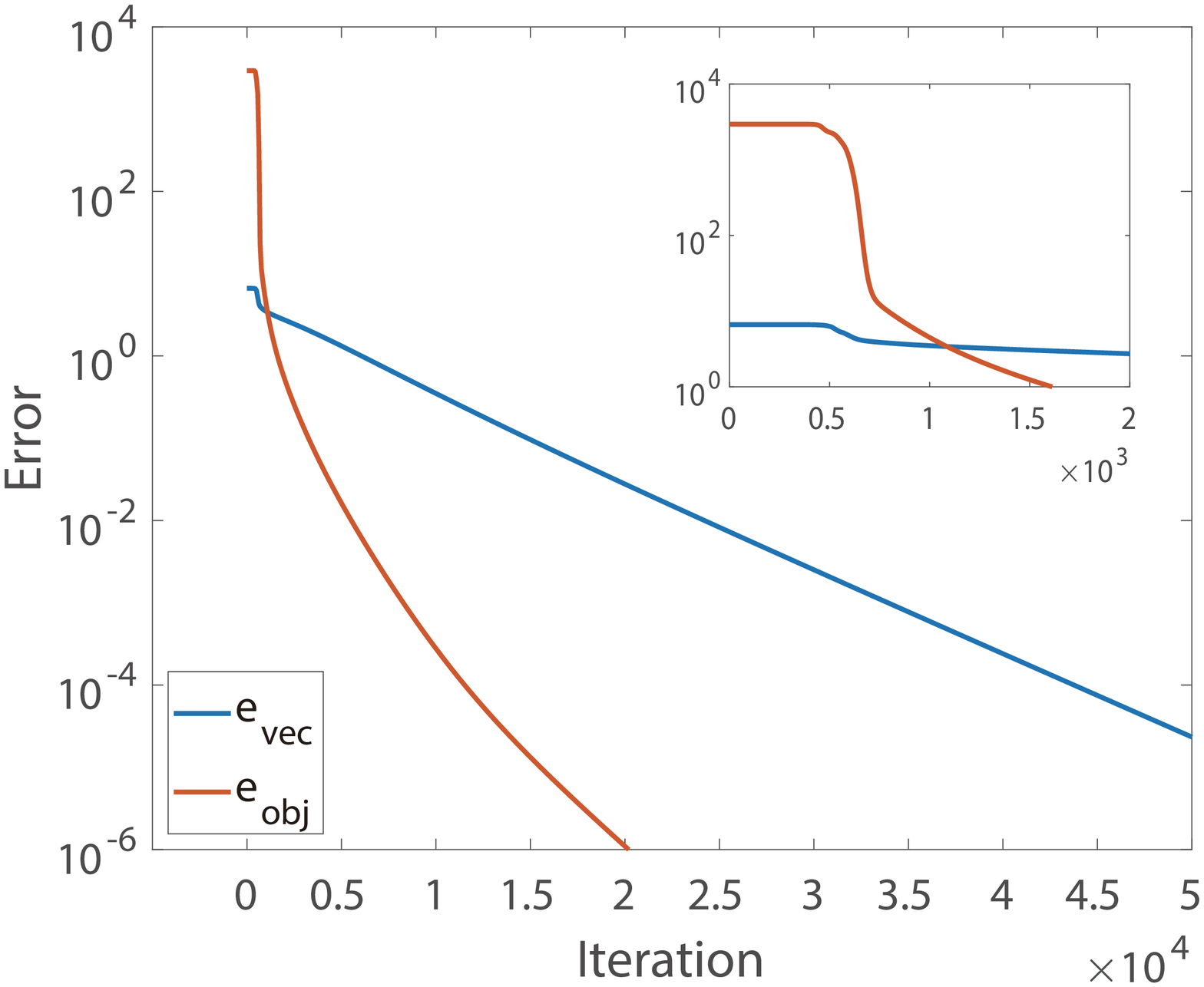}
    \caption{Convergence behavior of two tests, each are depicted with
    $e_{\text{vec}}$ and $e_{\text{obj}}$. The first few iterations
    are zoomed in on their top right corners.} \label{fig:conv_behave}
\end{figure}

Since the initial points are chosen near some saddle points, in the first
few iterations, the error decays slowly, and the linear convergence rate
is not guaranteed. After $X^{(t)}$ escapes from the saddle point near the
initial point, it may fall into the domain of local linear convergence
directly (Figure~\ref{fig:conv_behave} right) or slow down several times
before linear convergence (Figure~\ref{fig:conv_behave} left), which
depends on the choice of initial points. According to our analysis, the
slow-down happens a finite number of times, and the iteration will
converge monotonically to the global minima. All of our experiments,
including the two in Figure~\ref{fig:conv_behave} agree well with our
analysis.


\medskip
\noindent

{\bf Acknowledgments.} The authors thank Jianfeng Lu and Zhe Wang for
helpful discussions. YL is supported in part by the US Department of
Energy via grant de-sc0019449. WG is supported in part by National Science
Foundation of China under Grant No. 11690013, U1811461.

\bibliographystyle{apalike} \bibliography{library}

\appendix

\section{Proof of Lemma~\ref{lem:bounded-domain}}
\label{app:whole_bound}

\begin{proof}
It is sufficient to show that the condition holds for one iteration. In
order to simplify the notations, we denote $x_i^{(t)}$ and $x_i^{(t+1)}$
as $x_i$ and $\tx_i$ respectively. The norms of $x_i$ and $\tx_i$ are
denoted as $r_i$ and $\tr_i$ respectively.

The iteration in \eqref{eq:triofm} can be written as,
\begin{equation}
    \tx_i = x_i - \alpha A x_i - \alpha (\sum_{j=1}^{i}x_j
    x_j^\top) x_i
    = x_i - \alpha \tA x_i - \alpha x_i
    x_i^\top x_i.
\end{equation}
where $\tA = A + \sum_{j=1}^{i-1} x_j x_j^\top $. The norm square of
$\tx_i$ can be calculated as,
\begin{align} \label{eq:norm_sqr_tx}
    \tr_i^2 = \tx_i^\top \tx_i & = x_i^\top x_i - 2 \alpha
    \left(x_i^\top \tA x_i + (x_i^\top x_i)^2 \right) \notag\\
    & \quad +
    \alpha^2 \left( x_i^\top \tA^\top \tA x_i + (x_i^\top
    x_i)^3 + 2x_i^\top \tA x_i (x_i^\top x_i) \right) \notag\\
    & = r_i^2 - 2 \alpha
    \left(x_i^\top \tA x_i + r_i^4 \right) +
    \alpha^2 \left( x_i^\top \tA^\top \tA x_i + r_i^6 + 2
    x_i^\top \tA x_i r_i^2 \right).
\end{align}
Given that all $x_i$ satisfy the conditions $\norm{x_i}
\leq R_i$, we have inequality for any vector $x$ and power $k$,
\begin{equation} \label{eq:ineqA}
    -\left( \rho + \sum_{j=1}^{i-1}R_j^2 \right)^k \norm{x}^2
    \leq x^\top \tA^k x \leq \left( \rho + \sum_{j=1}^{i-1}R_j^2
    \right)^k \norm{x}^2,
\end{equation}
where we adopt the definition of eigenvalues for the first part in $\tA$
and Cauchy-Schwartz inequality for the second part in $\tA$. Due to the
assumption on $R_j$, we can bound the growing factor in \eqref{eq:ineqA}
as,
\begin{equation} \label{eq:ineqsum}
    \rho+\sum_{j=1}^{i-1}R_j^2
    \leq \rho \left( 1 + 3\sum_{j=1}^{i-1}4^{j-1} \right)
    \leq 4^{i-1}\rho = \frac{R_i^2}{3}.
\end{equation}

With these inequalities, the first order term of $\alpha$ in
\eqref{eq:norm_sqr_tx} can be bounded as,
\begin{equation}
    -2\alpha \left(x_i^\top \tA x_i + r_i^4 \right)
    \leq 2\alpha r_i^2
    \left( \frac{ R_i^2 }{3} - r_i^2 \right).
\end{equation}
And the second order term can be bounded as,
\begin{align} \label{eq:term2}
    \alpha^2 \left(x_i^\top \tA^2 x_i + r_i^6 +
    2x_i^\top \tA x_i r_i^2 \right) & \leq \alpha^2 \left(
    \left(\rho+\sum_{j=1}^{i-1}R_j^2 \right)^2 r_i^2 +
    r_i^6 + 2\left(\rho + \sum_{j=1}^{i-1} R_j^2\right)
    r_i^4 \right) \notag\\
    & \leq \alpha^2 r_i^2 \left(
    \frac{R_i^4}{9} +
    r_i^4 +  \frac{2 R_i^2}{3} r_i^2 \right).
\end{align}
The rest of the proof is divided into two scenarios, \ie,
$r_i \in [\frac{\sqrt{2}R_i}{\sqrt{3}}, R_i]$ and $r_i \in [0,
\frac{\sqrt{2}R_i}{\sqrt{3}})$.

In the first scenario, $r_i \in [\frac{\sqrt{2}R_i}{\sqrt{3}},
R_i]$, we have $ -2\alpha \left( x_i^\top \tA x_i + r_i^4
\right)\leq -\alpha r_i^2 \frac{2R_i^2}{3}$ for the first order
term.  Applying $\alpha \leq \frac{1}{5R_p^2}$, we have,
\begin{equation}
    \tr_i^2 \leq r_i^2 + \alpha r_i^2
    \left( -\frac{2R_i^2}{3} + \frac{R_i^4}{45 R_p^2} +
    \frac{R_i^4}{5 R_p^2} + \frac{2R_i^4}{15 R_p^2} \right)
    \leq r_i^2 - \alpha r_i^2
    \frac{14R_i^2}{45} < r_i^2 \leq R_i^2.
\end{equation}

In the second scenario, $r_i \leq \sqrt{\frac{2}{3}}R_i $, we have
\begin{align}
    \tr_i^2 & \leq r_i^2 + 2\alpha r_i^2 \left( \frac{ R_i^2 }{3}
    - r_i^2 \right) + \alpha^2 r_i^2 \left( \frac{R_i^4}{9}
    + r_i^4 +  \frac{2R_i^2}{3} r_i^2 \right) \notag\\
    & \leq \frac{2R_i^2}{3} + \frac{ 4 R_i^4 }{45 R_p^2}
    + \frac{2 R_i^6}{675 R_p^4}
    + \frac{ 8 R_i^6}{675 R_p^4} + \frac{8R_i^6}{675 R_p^4} <
    R_i^2. 
\end{align}

This proves the lemma.
\end{proof}

\section{Proof of Lemma~\ref{app:lemma8}}
\label{app:lemma8-proof}

\begin{proof}
Without loss of generality, we assume $A$ is diagonal. Notations remain
the same as that in the proof of Lemma~\ref{app:lemma5} if not redefined.
Based on the assumptions, we denote $k \in (i, q]$ as the smallest integer
such that $e_k^\top x_i^{(t)} = x_{ik}^{(t)}$ does not converge to zero.
Hence, there exists a positive number $\delta > 0$ such that for any $N$,
there exists a $t>N$ and $\abs{x_{ik}^{(t)}} > \delta$.  We also know that
Lemma~\ref{app:lemma6} holds, $\norm{E^{(t)}}$ converges to zero, and
$e_j^\top x_i^{(t)} = x_{ij}^{(t)}$ converges to zero for all $j \in [1,i)
\cup (p,n]$. Hence, for any $\veps$ sufficiently small, there exists an
integer $N_1$ such that $\norm{E^{(t)}} \leq \veps^2$, $\norm{x_i^{(t)}}
\geq \frac{\sqrt{-2\lambda_q}}{4}$, and $\abs{x_{ij}^{(t)}} \leq \veps $
hold for all $t > N_1$ and $j \in [1, i) \cup (p,n]$. Since $e_j^\top
x_i^{(t)} = x_{ij}^{(t)}$ does not converge to zero, there exists an
integer $N_2 > N_1$, such that,
\begin{equation} \label{eq:ratio-cond-j}
    \frac{\abs{x_{ij}^{(N_2)}} }{\norm{x_{i}^{(N_2)} }}
    \geq \frac{\delta}{R_i} \geq
    \frac{ 2 \veps^2}{\lambda_{j+1}-\lambda_j}.
\end{equation}

Using the iterative relationship \eqref{eq:it-for-xi}, we have the
lower bound on $\abs{x_{ij}^{(N_2+1)}}$,
\begin{equation} \label{eq:denominator-lower-bound-j}
    \abs{x_{ij}^{(N_2+1)}} \geq \left(1 - \alpha \frac{\lambda_j
    + \lambda_{j+1}}{2} - \alpha \norm{x_i^{(N_2)}}^2 \right)
    \abs{x_{ij}^{(N_2)}}.
\end{equation}

Again using the iterative relationship \eqref{eq:it-for-xi},
we provide the upper bound for
\begin{align} \label{eq:numerator-upper-bound-j}
        \sqrt{ \sum_{j \neq k} \left( x_{ij}^{(N_2+1)}
        \right)^2 } & \leq \Bigg\{ \sum_{j < i} \left(1 - \alpha
        \norm{x_i^{(N_2)}}^2\right)^2 \left(x_{ij}^{(N_2)}\right)^2
        \notag\\
        & \quad + \sum_{i \leq j < k} \left(1 - \alpha
        \lambda_j - \alpha \norm{x_i^{(N_2)}}^2\right)^2
        \veps^2 \notag\\
        & \quad + \sum_{j > k} \left(1 - \alpha \lambda_{k+1} - \alpha
        \norm{x_i^{(N_2)}}^2\right)^2 \left(x_{ij}^{(N_2)}\right)^2
        \Bigg\}^{\frac{1}{2}} + 3 \sqrt{\alpha n}R_i \veps \notag\\
        & \leq \left(1 - \alpha \lambda_{k+1} - \alpha
        \norm{x_i^{(N_2)}}^2\right) \sqrt{ \sum_{j \neq k}
        \left(x_{ij}^{(N_2)}\right)^2 } + \sqrt{n}\veps + 3
        \sqrt{\alpha n}R_i \veps,
\end{align}
where the derivation is slightly
different from that in \eqref{eq:numerator-upper-bound},
The first inequality adopts similar derivation in
\eqref{eq:numerator-upper-bound} while keeps the first term
unchanged; and the second inequality mainly uses the inequality
of square-root function.

Substituting \eqref{eq:denominator-lower-bound-j} and
\eqref{eq:numerator-upper-bound-j} into the expression of $\tan
\theta_k^{(N_2+1)}$, we obtain,
\begin{equation} \label{eq:tan-ineq-j}
    \tan \theta_k^{(N_2+1)} \leq (1 - \beta) \tan \theta_k^{(N_2)}
    - \beta \tan \theta_k^{(N_2)} + C \veps,
\end{equation}
where $\beta = \frac{\alpha (\lambda_{i+1} - \lambda_i)}{4
- 2\alpha (\lambda_i + \lambda_{i+1})} \in (0,1)$ and $C =
\frac{8\sqrt{n}R_i + 24\sqrt{\alpha n}R_i^2}{\sqrt{-2\lambda_q}
\delta}$.

Based on the last inequality in \eqref{eq:tan-ineq-j},
if $\tan \theta_k^{(N_2)} > \frac{C\veps}{\beta}$,
than we have $\tan \theta_k^{(N_2+1)} < (1 - \beta) \tan
\theta_k^{(N_2)}$, which implies $\cos \theta_k^{(N_2+1)} >
\cos \theta_k^{(N_2)}$ due to the fact that all angles acute
angle. Therefore, \eqref{eq:ratio-cond-j} holds for $t =
N_2+1$ and $\tan \theta_k^{(t)}$ decay monotonically until
$\tan \theta_k^{(t)} \leq \frac{C \veps}{\beta}$. When $\tan
\theta_k^{(t)} \leq \frac{C \veps}{\beta}$, we obviously have
$\tan \theta_k^{(t+1)} \leq \frac{C\veps}{\beta}$. The inequality
condition \eqref{eq:tan-ineq} still holds as long as $\veps$ is
sufficiently small.  Hence there exists a $N$ such that for all $t
> N$, we have $\tan \theta_k^{(t+1)} \leq \frac{C\veps}{\beta}$,
which can be arbitrarily small.
\end{proof}

\section{Proof of Lemma~\ref{lem:same-saddle}}
\label{app:same-saddle-proof}

\begin{proof}
This lemma is proved in a similar way as that of
Lemma~\ref{lem:op-saddle}. If the lemma does not hold, there is some
$t^\prime \in (t_1, t_2)$ such that $\norm{x_i^{\left(t^\prime\right)}
-s}\geq 6\sqrt{n} \veps_i^{\frac{1}{3}}$. Let $t_1^\prime \in (t_1,
t^\prime]$ be the latest time such that $x_i^{\left(t_1^\prime-1\right)}
\in \calN$ and $x_i^{\left(t_1^\prime\right)} \not \in \calN$ for $\calN =
\calN_{s,\veps_i}$. Also let $t_2^\prime \in  (t^\prime, t_2]$ be the
first time such that $x_i^{\left(t_2^\prime \right)} \in \calN$. Then we
have,
\begin{align} \label{eq:same-saddle-pts}
    \norm{\grad{F\left(x_i^{\left(t_1^\prime-1\right)}\right)}} < \veps_i, \quad
    \norm{x_i^{\left(t_1^\prime-1\right)} - s } < \sqrt{n} \veps_i^{\frac{1}{3}}, \quad
    \norm{\grad{F\left(x_i^{\left(t_1^\prime\right)}\right)}} \geq \veps_i,\notag\\
    \norm{\grad{F\left(x_i^{\left(t_2^\prime\right)}\right)}} < \veps_i, \quad
    \norm{x_i^{\left(t_2^\prime\right)} - s} < \sqrt{n}\veps_i^{\frac{1}{3}}.
\end{align}

By the construction of $t^\prime_1$ and $t_2^\prime$, we know that for all
$t \in [t^\prime_1, t_2^\prime)$, $x_i^{(t)} \not\in \calN$. Adopting the
same derivation as in \eqref{eq:paradox1}, we obtain a lower bound on the
function difference,
\begin{equation}
    F\left(x_i^{\left(t_1^\prime\right)}\right)
    - F\left(x_i^{\left(t_2^\prime\right)}\right)
    \leq 3\sqrt{n}\veps_i^{\frac{4}{3}}.
\end{equation}

On the other hand, lower bound of the energy function decrement for
$x_i^{\left(t_1^\prime\right)}, \dots, x_i^{\left(t_2^\prime\right)}$ can
be estimated as, according to Lemma~\ref{lem:en_decay},
\begin{align}
        F\bigg(x_i^{\left(t_1^\prime\right)}\bigg)
        - F\bigg(x_i^{\left(t_2^\prime\right)}\bigg)
        & > \frac{1}{2} \veps_i \left( \sum_{t=t_1^\prime}^{t^\prime-1}
        \left(\norm{x_i^{(t)}-x_i^{(t+1)}}\right)+\sum_{t=t^\prime}^{t_2^\prime}
        \left(\norm{x_i^{(t)}-x_i^{(t+1)}}\right)\right) \notag\\
        & \geq \frac{1}{2}\veps_i\Bigg(\left\|x_i^{\left(t_1^\prime-1\right)}
        -x_i^{\left(t^\prime\right)}\right\|+\left\|x_i^{\left(t_2^\prime\right)}
        -x_i^{\left(t^\prime\right)}\right\|-\norm{x_i^{\left(t_1^\prime-1\right)}
        -x_i^{\left(t_1^\prime\right)}}\Bigg)\notag\\
        & > 4\sqrt{n} \veps_i^{\frac{4}{3}},
\end{align}
where the last inequality is based on \eqref{eq:same-saddle-pts} and
$\norm{x_i^{\left(t^\prime\right)}-s}\geq 6\sqrt{n}
\veps_i^{\frac{1}{3}}$. Thus the contradiction is derived.
\end{proof}

\end{document}